\documentclass[10pt]{amsart}
 \usepackage{amsmath}
\usepackage{amssymb}
\usepackage{amsfonts}
 \usepackage{eucal}
      \makeatletter
     \def\section{\@startsection{section}{1}%
     \z@{.7\linespacing\@plus\linespacing}{.5\linespacing}%
     {\bfseries%\normalfont\scshape
     \centering
     }}
     \def\@secnumfont{\bfseries}
     \makeatother
 \usepackage{a4wide}
\newtheorem{theorem}{Theorem}[section]
\newtheorem{lemma}[theorem]{Lemma}
\newtheorem{corollary}[theorem]{Corollary}
\newtheorem{proposition}[theorem]{Proposition}

\theoremstyle{definition}

\newtheorem{problem}[theorem]{Problem}

\theoremstyle{remark}
\newtheorem{remark}[theorem]{\bf Remark}

\numberwithin{equation}{section}

\newcommand{\M}{{\mathcal M}}

\newcommand{\E}{{\mathbb E}}

\newcommand{\N}{{\mathbb N}}
\newcommand{\R}{{\mathbb R}}

\newcommand{\Z}{{\mathbb {Z}}}
\newcommand{\e}{{\rm e}}

\numberwithin{equation}{section}
\def \a{{\alpha}}

\def \d{{\delta}}

\def \g{{\gamma}}

\def \k{{\kappa}}
\def \l{{\lambda}}

\def \o{{\omega}}
\def \O{{\Omega}}
\def \p{{\varphi}}

\def \m{{\mu}}
\def \s{{\sigma}}

\def \qq{{\qquad}}
\def \noi{{\noindent}}

\def\E{{\mathbb E \,}}

\def\R{{\mathbb R}}

\def\Z{{\mathbb Z}}

\def\N{{\mathbb N}}

\title[Ergodic theorems with arithmetical weights]{\bf Ergodic theorems with arithmetical weights}

\author{Christophe Cuny}
\address{Ecole Centrale, Grande Voie des Vignes, 92290 Ch\^atenay-Malabry, France}
\email{christophe.cuny@ecp.fr}

\author{Michel Weber}

\address{IRMA, 10 rue du G\'en\'eral Zimmer,
67084 Strasbourg Cedex, France}
\email{michel.weber@math.unistra.fr}

\begin{document}

\begin{abstract} We prove that the divisor function $d(n)$ counting the number of 
divisors of  the integer $n$, is a good weighting function for the pointwise ergodic theorem. For    any measurable dynamical system $(X, {\mathcal A},\nu,\tau)$    and any $f\in L^p(\nu)$, $p>1$, the limit
$$
\lim_{n\to \infty}{1\over \sum_{k=1}^{n} d(k)} \sum_{k=1}^{n} d(k)f(\tau^k x)$$
exists $\nu$-almost everywhere. We also obtain similar results for other arithmetical functions, like    $\theta(n)$ function counting the number of squarefree  divisors of $n$
 and the generalized Euler totient function $J_s(n)
 %=\sum_{d|n} d^s \m(\frac{n}{d})
 $, $s>0$. We use Bourgain's method, namely the circle method based on the shift model.

\end{abstract}

%\date{}
\maketitle
 
\section{Introduction} \label{intro}Let $(X, {\mathcal A},\nu,\tau)$ be a measurable dynamical system. Birkhoff's pointwise ergodic theorem states that for any $f\in L^1(\nu)$, the
limit
$$
\lim_{n\to \infty}{1\over n} \sum_{k=0}^{n-1} f(\tau^k x)
  =\bar f(x)
$$
exists $\nu$-almost everywhere and in $L^1(\nu)$,
%; further $\bar f= \E\{ f|\mathcal J\}(x)$ where $\mathcal J =\s \{ A\in \mathcal A :
%\tau^{-1}A=A\}$.  In
and  $\bar f=\int fd\nu $  if    $\tau$
is ergodic.
This fundamental result was the object of many generalizations or extensions. We are  interested in this article in extensions of weighted type,
more particularly in extensions in which the weights are built with standard arithmetical functions, typically the divisor function $d(n)$, counting the number of divisors of  the integer $n$.
This is the most standard example of multiplicative arithmetical function, but the reason to focus on this particular type of weights  lies  on deeper considerations.

First introduce the necessary  notation. For a sequence $(w_k)_{k\ge 1}$ of
real numbers (weights)  such that
$W_n:= \sum_{k=1}^n |w_k|\neq0$ and  $W_n\to \infty$,
we define  the weighted averages
$$
A_n^\tau f:= \frac1{W_n}\sum_{k=1}^n w_k f\circ \tau^k . 
$$
We are interested in their almost everywhere convergence. \vskip 3 pt
 As we will see later on (Theorems \ref{cordelange},  \ref{corcol},  \ref{listex}, Proposition \ref{momax} and Remark \ref{remdelange}), elementary considerations based on Hopf's maximal Lemma or on properties of  Dirichlet convolution
products  allow to directly derive from  Birkhoff's theorem, weighted extensions where the weights  can be either of the examples below
\begin{eqnarray*}
 \o(k), \O(k) &&\qq\hbox{\rm the prime divisor function and the sum of prime divisor function,}
\cr \s_s(k) &&\qq\hbox{\rm the sum of $s$-powers of divisors of $k$, $s\not=0$.}
% \cr   J_s(k) &&\qq\hbox{\rm the generalized Euler totient function, $s>0$.}
\end{eqnarray*}   
The case of the divisor function  corresponding to the second example with $s=0$, does not seem to be
 reduced to this approach. 
% The other reason 
 Another motivation lies in a recent result of Berkes, M\"uller and Weber \cite{BMW}, Theorem 3.
\vskip 4 pt \noi {\bf Theorem A.} {\it Let $f$ be a non-negative multiplicative function and let $F(n) =\sum_{k=1}^n f(k)$, $n\ge 1$. Assume that there are
positive constants $C_1, C_2, C_3, C_4$ and $a > 1$ such that
\begin{eqnarray*} (i) \, \sum_{p\le x} f(p)^a \log p\le C_1 x ,\quad  
(ii)\,  \sum_{p,\nu\ge 2}\frac{f(p^\nu)^a  \log p^\nu}{p^{ \nu}}   \le C_2, \quad
(iii)\, \sum_{p\le x} f(p)\log p\ge C_3 x\  \hbox{for $x\ge C_4$}. \end{eqnarray*}
 %\begin{eqnarray*} \begin{cases}\displaystyle{\sum_{p\le x} f(p)^a \log p\le C_1 x},\cr 
%\displaystyle{\sum_{p,\nu\ge 2}f(p^\nu)^a p^{-\nu} \log p^\nu\le C_2},\cr
%\displaystyle{ \sum_{p\le x} f(p)\log p\ge C_3 x\qq \hbox{for $x\ge C_4$}.}\qq \qq  \end{cases}
%\end{eqnarray*}
Let $X_1,X_2, \ldots$ be i.i.d. integrable random variables.Then}
 $
\lim_{n\to \infty} \frac1{F(n)}\sum_{k=1}^n f(k)X_k\buildrel{a.s.}\over{=}\E X_1$. 
\vskip 4 pt  As a consequence, the weighted strong law of large numbers holds with $f(n)=d(n)$ (and even for $f(n)= d(n^l)^\a$ with $l\ge 1$ integer and real $\a>0$). Theorem A is obtained by showing that Jamison, Orey and Pruitt combinatorial criterion is satisfied under the above set of conditions.  
%The criterion was also applied in   Lin and Weber   \cite{LW}.

\vskip 2 pt  Here, we use Bourgain's method  to study the case of the divisor function. As  Bourgain mentionned in \cite{B1}, his method  (the circle method on the shift model) is general and should apply as soon as one has a good control on the exponential sums inherent to the problem. One can then replace the Fourier kernels appearing in the problem considered by  more regular ones with a suitable control on  the error term.  
\vskip 3 pt 
%Introduce a definition. 
We say that $(w_k)_{k\ge 1}$ is a {\it good weight for the dominated ergodic theorem in 
$L^p$}, $p>1$, if there exists $C_p>0$ such that for every (ergodic) 
dynamical system $(X, {\mathcal A},\nu,\tau)$ and every $f$ in $L^p$, 
$$
\Big\|\sup_{n\ge 1}\frac{|\sum_{1\le k\le n} w_kf\circ \tau^k|}{W_n}
\Big\|_p\le C_p\|f\|_p\, .\
$$
We also say that $(w_k)_{k\ge 1}$ is a {\it good weight for the pointwise ergodic theorem in 
$L_p$}, $p>1$, if for every (ergodic) 
dynamical system $(X, {\mathcal A},\nu,\tau)$ and every $f$ in $L^p$, 
$\big((\sum_{1\le k\le n} w_kf\circ \tau^k)/W_n\big)_n$ converges 
$\nu$-a.s.  Alternatively, when the weights are generated by an arithmetical function $w$, we say that 
$w$ is    a {\it good weighting  function}. \vskip 3 pt We now state our main result. 
%Let $D(x) =\sum_{ k\le x } d (n)$ and recall   that
%$$D(x) =  x(\log x +2\g-1)+{\mathcal O}(\sqrt x), $$
% where $\g$ is Euler's constant. 
\begin{theorem}\label{mt} The divisor function    is a good weighting  function for both  the dominated 
and the pointwise ergodic theorem in $L_p$, $p>1$.
%Let $(X, {\mathcal A},\m,
%\tau)$ be a measurable dynamical system. For every $p>1$, there exists 
%$C_p>0$ such that for every $f\in L^p(\mu)$, 
%\begin{equation*}
%\Big\|\sup_{n\ge 1}\Big |\frac{\sum_{1\le k\le n} d(k) f\circ \tau^k}{D(n)}\Big| \,\Big\|_{p,\mu}\le C\|f\|%_{p,\nu}\, .
%\end{equation*}
\end{theorem}
Using properties of the Dirichlet convolution, it will also follow that the above result remains true for other arithmetical
weights. Let
$\theta(n)$ be the multiplicative function counting the number  of squarefree divisors of $n$, and let 
% and let $\Theta(x) =\sum_{ k\le x }   \theta(k)$. 
%Notice that $\theta(k)= 2^{\o(k)} $,  
%$\o(k)$ being the prime  divisor function, namely counting the number of prime divisors of  $k$; 
%and   that the set of squarefree integers has positive density
% (in the usual sense)
%  equal to $6/\pi^2$.
   $J_s(n)$ be the generalized Euler totient function. Recall that   $\theta(k)= 2^{\o(k)} $ and 
$J_s(n)=\sum_{d|n} d^s \m(\frac{n}{d})$ where $\m$ is M\"obius function. 
%Recall also (\cite{MV} p.42) that,  $$\Theta(x)  =\frac{6}{\pi^2} x
%\log x  +x\Big(2\g-1+ 2\frac{\zeta'(2)}{\zeta(2)^2}\Big) +{\mathcal O}(x^{1/2}\log x) \qq x\ge 2.$$  
\begin{theorem}\label{mt1}
 The $\theta $ function and  generalized Euler totient function $J_s$, $s>0$, the sum of  $s$-powers of divisor function $\s_s$, $s\neq 0$ and the modulus of M\"obius function  are good weighting  functions for both  the dominated 
and the pointwise ergodic theorem in $L_p$, $p>1$.
%Let $(X, {\mathcal A},\m,
%\tau)$ be a measurable dynamical system. For every $p>1$, there exists 
%$C_p>0$ such that for every $f\in L^p(\mu)$, 
%\begin{equation*}
%\Big\|\sup_{n\ge 1}\Big |\frac{\sum_{1\le k\le n} \theta(k) f\circ \tau^k}{\Theta(n)}\Big| \,\Big\|_{p,\mu}%\le C\|f\|_{p,\nu}\, .
%\end{equation*}
\end{theorem}
 The paper is organized as follows. In Section \ref{derivb}, we derive  from the dominated 
and the pointwise ergodic theorem (in $L^p$, $p>1$) several   weighted ergodic theorems where the weights are mainly additive arithmetical functions.  We use a Theorem of Delange \cite{Delange} and the Tur\'an-Kubilius inequality (see e.g. \cite{Tenenbaum}), as well as the result of Davenport  and of Bateman and Chowla for the case of the M\"obius and Liouville  functions respectively. 
\vskip 2 pt 
In Section \ref{Dirconv}, we consider the following problem. Given  a good weighting function $a  $  and another arithmetical function $b $, we study  the conditions under which the Dirichlet convoluted function $a*b $ is   again a good weighting function. We recall that $a*b $ is defined by $a*b(n)=\sum_{d|n} a(d)b(n/d)$.
After  having first proved some   lemmas a bit in the spirit of Wintner's theorem, we obtain  in Proposition \ref{cgw},  a general condition showing a kind of conservation property for the dominated ergodic theorem under the action of the Dirichlet convolution product.     We   next  apply it and show   that  the sum of 
$s$-powers of divisors of $k$, $s\not=0$, the number of squarefree  divisors of $k$, the generalized Euler totient function and    the
modulus of M\"obius function    are  good weighting functions  
 for the dominated ergodic theorem in 
$L^p$, $p>1$, and  for the pointwise ergodic theorem in $L^p$, $p\ge 1$.
  \vskip 2 pt In section \ref{sketchbo},  Bourgain's approach is briefly described, essentially the key steps are presented. In  section \ref{divest}, several   estimates concerning 
the divisor exponential sums $D_n(x)  =  \sum_{1\le k\le n} d(k) \e^{2ik\pi x}$ are proved, depending on the proximity of $x$ to rationals with small or large denominators. We proceed rather simply and will   not for instance
use Vorono\"\i \, identity, nor  need elaborated estimates.  
\vskip 2 pt In the two next sections, we apply Bourgain's approach. In Section \ref{fourierineq}, we use Fourier analysis to establish maximal inequalities related to  auxiliary kernels. 
 In   Section \ref{app}, we explain how to derive  good approximation results with  suitable Fourier kernels to which we can apply the previous maximal inequalities.  Theorem \ref{mt} is proved in section \ref{proofmt}. 

 %And in section \ref{gendivest}, we extend this result to generalized divisor %functions. 

%%%%%%%%%%%%%%%%%%%%%%%%%%%%%%%%%%%
%%%%%%%%%%%%%%%%%%%%%%%%%%%%%%%%%%%
\section{\bf   Derivation Results  from Birkhoff's Theorem.} \label{derivb}
 For a sequence $(a_n)_{n\ge 1}$, write $A_n^{(m)}= 
\sum_{k=1}^n a_k^m$, for every real number $m>0$. Define 
also $\tilde A_n:= A_n/n$.
  The following is a basic application of H\"older inequality, hence 
the proof is omitted. 

\begin{lemma}\label{simplelemma}
Let $(a_n)_{n\ge 1}$ be a sequence of non-negative numbers. 
Assume that there exists $C>0$ and $m>1$ such that for every 
$n\ge 1$, $\sum_{k=1}^n a_k^m\le C n \tilde A_n^m$. Then, for every 
$1\le s\le m$ and every $n\ge 1$, $\sum_{k=1}^n a_k^s\le 
C_{m,s} n \tilde A_n^s$, with $C_{m,s}=C^{m(s-1)/(s(m-1))}$.
\end{lemma}
\vskip 2 pt
We deduce  the following automatic dominated ergodic theorem. 

\begin{lemma}\label{simplelemma2} 
Let $(a_n)_{n\ge 1}$ be a sequence of non-negative numbers.  
Assume that there exists $C>0$ and $m>1$ such that for every 
$n\ge 1$, $\sum_{k=1}^n a_k^m\le C n \tilde A_n^m$. Then, 
for every $r>m/(m-1)$,  $(a_n)_{n\ge 1}$ is good for the dominated 
ergodic theorem in $L^r$. 
\end{lemma}
\begin{remark} It follows from the proof that, for $r=m/(m-1)$, 
there is a dominated ergodic theorem of weak-type.
\end{remark}
\begin{proof} Let 
$(X,{\mathcal A},\nu,\tau)$ be a dynamical system and let $f\in L^r(\nu)$, 
for some $r\in (m/(m-1),+\infty]$. Then, $r/(r-1)<m$.
Now let $r/(r-1)<s < m$. Then, writing $s':=s/(s-1)$, we have 
$ m/(m-1)<s'< r$. 
\vskip 2 pt 
By Lemma \ref{simplelemma}, 
we have $\sum_{k=1}^n a_k^s\le C n\tilde A_n^s$. By H\"older inequality, we have 
\begin{eqnarray*}
\big|\sum_{k=1}^n a_kf\circ \tau^k\big|& \le &
\big(\sum_{k=1}^n a_k^s\big)^{1/s}\, \big( \sum_{k=1}^n  |f\circ \tau ^k|
^{s'}\big)^{1/s'} 
\ \le \ C^{1/s} n^{1/s}\tilde A_n\big( \sum_{k=1}^n  |f\circ \tau ^k|
^{s'}\big)^{1/s'}\\
&= & C^{1/s}A_n \Big( \frac1n\sum_{k=1}^n  |f\circ \tau ^k|
^{s'}\Big)^{1/s'}\, .
\end{eqnarray*}
\end{proof}
\begin{corollary}\label{cor}
Let $(a_n)_{n\ge 1}$ be a sequence of non-negative numbers.  
Assume that there exists  $m>1$ such that  
$$\sum_{k=1}^n |a_k-\tilde A_n|^m=o\big( n \tilde A_n^m\big)\, . 
$$Then, 
for every $r>m/(m-1)$,  $(a_n)_{n\ge 1}$ is good for both  the dominated 
and the pointwise ergodic theorem in $L^r$. 
\end{corollary}
\begin{remark} It follows from the proof that, for $s=m/(m-1)$, 
there is a dominated ergodic theorem of weak-type and a pointwise 
weighted ergodic theorem.
\end{remark}
\begin{proof}
Using that $a_k^m\le 2^{m-1} (\tilde A_n^m+ |a_k-\tilde A_n|^m)$, we see that there exists $C>0$ such that for every $n\ge 1$  $\sum_{k=1}^n a_k^m\le C n \tilde A_n^m$, and Lemma \ref{simplelemma2} applies, as well
 as the remark after it. By the Banach principle, we just have to 
 prove the pointwise convergence for bounded functions.   Let 
$(X,{\mathcal A},\nu,\tau)$ be a dynamical system and let $f\in 
L^\infty(\nu)$. Let $K\ge 0$ be such that $f\le K$ $\nu$-a.s.

\vskip 2 pt

We have, by H\"older 
\begin{gather*}
|\sum_{k=1}^n (a_k-\tilde A_n)f\circ \tau^k|\le 
K n^{1-1/m}(\sum_{k=1}^n |a_k-\tilde A_n|^m)^{1/m}=o(n\tilde A_n)\, ,
\end{gather*}
and the result follows. 
\end{proof}

\begin{theorem}\label{cordelange}
Let $(g(n))_{n\ge 1}$ be an additive function with values in $\N$ and 
such that $g(p)=1$ for every prime number $p$. Assume moreover that there exists $\beta>0$, such that for every $\nu\ge 1$ and every prime number $p$, 
\begin{equation}\label{h} g(p^\nu)\le \beta \nu \log p.\end{equation} Then 
$(g(n))_{n\ge 1}$ is a good weight for both the dominated and pointwise 
ergodic theorem in $L^p$, $p>1$.
\end{theorem}
\begin{remark} \label{remdelange} It follows from the proof that for every 
real number  $m\ge 1$, $(g(n)^m)_{n\ge 1}$ is a good weight 
for the dominated ergodic theorem in $L^p$, $p>1$. When $m$ is an integer, it is also a good weight for the pointwise ergodic theorem in $L^p$, 
$p>1$. The theorem applies in particular with $g(n)=\omega(n)$ and 
$g(n)=\Omega(n)$.
\end{remark}
\begin{proof} Let us recall the following corollary of a deep 
result of Delange \cite{Delange}. The corollary corresponds to 
Theorem 2 (p. 132) with $\nu=m$ and $\chi\equiv 1$, provided that (9) 
in \cite{Delange} be satisfied. We shall check this below.

\begin{theorem}\label{delange}
Let $(g(n))_{n\ge 1}$  be as in Corollary \ref{cordelange}. 
For every integer $m \ge 1$, we have
\begin{equation*} 
\sum_{1\le n\le x} g(n)^m =x (\log \log x)^m + O(x(\log \log x)^{m-1}))\,. 
\end{equation*}
\end{theorem}

We see that the assumptions of Lemma \ref{simplelemma} (hence of 
Lemma \ref{simplelemma2}) are satisfied for every integer $m\ge 1$, 
with $\alpha=0$. Hence we have the dominated ergodic theorem. 
\vskip 2 pt 
Let us prove   the pointwise convergence of the weighted averages. 
It suffices to prove the convergence for bounded functions. 
Let $(X,{\mathcal A},\nu,\tau)$ be a dynamical system. 
Let $f\in L^\infty (\nu)$, with $|f|\le A$.  
 We agree to denote here and in what follows $ \log\log x =\log(\log(2+x))$, $x\ge 1$. We have 

\begin{gather}\label{simdec}
\sum_{n=1}^N g(n)f\circ \tau^n =
(\log\log  N) \sum_{n=1}^N f\circ \tau ^n + 
 \sum_{n=1}^N(g(n)-\log\log N)f\circ \tau^n\, .
\end{gather}
 By Theorem \ref{delange} and Birkhoff's ergodic theorem, 
$\frac{\log\log  N}{\sum_{1\le k\le N}g(k)}\sum_{1\le k\le N} f\circ 
\tau^k$ converges $\nu$-a.s. To conclude it suffices to prove that the second term in \eqref{simdec} converges $\nu$-a.s to 0.
\vskip 2 pt

By Cauchy-Schwarz's inequality, we have 
$$
\big|\sum_{n=1}^N(g(n)-\log\log  N)f\circ \tau^n\big| 
\le A \sqrt N \big(\sum_{n=1}^N (g(n)-\log\log  N)^2\big)^{1/2}
$$
  Using Theorem \ref{delange} with $m=1$ and 
$m=2$ and $(g(n)-\log \log N)^2= g(n)^2 -2g(n)\log \log N +(\log \log N)^2$, we see that there exists $C>0$ such that
$$
|\sum_{n=1}^N(g(n)-\log\log  N)f\circ \tau^n| 
\le C\sqrt N ( N \log \log N)^{1/2}=o(\sum_{1\le n\le N}g(k))\, ,
$$
and the proof is completed. 
\end{proof}
\vskip 2 pt
Let us prove under (\ref{h}) that the condition $(9)$ of \cite{Delange} is satisfied. 
We have to prove that there exists $\rho>1$ and $\sigma<1$ such that 
$\sum_{k\ge 2, p\in \mathcal P} \frac{\rho^{g(p^k)}}{p^{\sigma k}}<\infty$. 
Take $\sigma =3/4$ and $\rho>1$ such that $\gamma:=2(\sigma -\beta\log \rho)
>1$. Notice that 
$$
\sum_{k\ge 2} \frac{\rho^{g(p^k)}}{p^{\sigma k}} 
\le \sum_{k\ge 2} \big(\frac{\rho^{\beta \log p}}{p^{\sigma  }}\big)^k
\le \frac1{p^\gamma} \frac{1}{1-1/p^{\gamma/2}}
\le \frac1{p^\gamma} \frac{1}{1-1/2^{\gamma/2}}\, ,
$$
and the desired result follows.

\begin{theorem}\label{corcol}
Let $g(n)_{n\ge 1}$ be an additive function such that 
$$
\Big(\sum_{p^\alpha\le n} \frac{g(p^\alpha)^2}{p^\alpha}\Big)^{1/2}=o \Big(
\sum_{ p^\alpha\le n} \frac{g(p^\alpha)}{p^\alpha}\Big)\, .
$$
Then, $(g(n)_{n\ge 1}$ is good for both the dominated  and 
the pointwise ergodic theorem in $L^p$ for every $p>2$. 
\end{theorem}
\begin{proof} Recall the  Tur\'an-Kubilius inequality \cite{Tenenbaum} p.\,302. There exists an absolute constant $\widetilde C$  such that for any additive complex-valued arithmetic function $f$,  
\begin{equation*}   \frac{1}{n} \sum_{k=1}^n\big|f(k) -\sum_{p^\alpha\le n}\,  \frac{f(p^\alpha)}{p^\alpha(1-p^{-1})}\big|^2\le   \widetilde C\sum_{p^\alpha\le n} \frac{|f(p^\alpha)|^2}{p^\alpha }, \qq (n\ge 2).
\end{equation*}
%And by Collison's Theorem (\cite{Co}, p.310),  \begin{equation*}  \sum_{k=1}^n\big|f(k) -
%\sum_{p^\alpha\le n}\,  \frac{f(p^\alpha)}{p^\alpha(1-p^{-1})}\big|^2\le   Cn\, \sum_{p^
%\alpha\le n} \frac{|f(p^\alpha)|^2}{p^\alpha }\qq (n\ge 4).\end{equation*}
Let $\gamma(n)=\sum_{p^\alpha\le n}\,  \frac{g(p^\alpha)}{p^\alpha(1-p^{-1})}$. By Cauchy-Schwarz's inequality, next Tur\'an-Kubilius inequality,
\begin{equation*}\Big|\frac{1}{n}\sum_{k=1}^n\big(g(k) -\gamma(n)\big)\Big|^2\le \frac{1}{n}\sum_{k=1}^n\big|g(k) -\gamma(n)\big|^2\le   C\sum_{p^\alpha\le n} \frac{|g(p^\alpha)|^2}{p^\alpha }=o\big(|\gamma(n)|^2\big).
\end{equation*}
In particular, $\widetilde G(n)=\frac{1}{n} \sum_{k=1}^ng(k)= \gamma(n)+H$, where $H=o(|\gamma(n)|)$. Writing $H=h\gamma (n)$ with $|h|\le 1/2$, if $n$ is large, we have $|\gamma (n)|\le |\widetilde G(n)|/(1-|h|)\le 2|\widetilde G(n)|$, and by Minkowski's inequality,
\begin{equation*} \Big(\frac{1}{n}\sum_{k=1}^n\big|g(k) -\widetilde G(n)\big|^2\Big)^{1/2}=o\big(|\gamma(n)|\big)=o\big(|\widetilde G(n)|\big).
\end{equation*}
We conclude by applying Corollary 
\ref{cor}.
\end{proof}

\vskip 3 pt \noi {\bf  The case of M\"obius and Liouville functions.}  Here we consider the $\nu$-a.s. behaviour of the 
sums $\sum_{k=1}^n \mu(k)f\circ \tau^k$,  $\sum_{k=1}^n \l(k)f\circ \tau^k$ where $\mu$ is 
the M\"obius function and $\l$ is Liouville function. We only treat the case of the M\"obius function, the arguments being quite identical for the  Liouville function.

\medskip

Let us recall the following result of Davenport \cite{Davenport}
on the behaviour of the corresponding exponential sums. 

\begin{proposition}\label{davenport}
For every $h>0$ there exists $C_h>0$ such that 
\begin{equation*}
\sup_{x\in [-1/2,1/2]} \big|\sum_{k=1}^n \mu(k) {\rm e}^{2i\pi kx}\big|
\le \frac{C_hn}{(\log n)^h}\, .
\end{equation*}
\end{proposition}
\begin{remark} According to  Lemma 1 in Bateman and Chowla \cite{BC}, an analog estimate holds for the  Liouville function.
\end{remark}
 By the spectral theorem (see e.g. \cite{W}, Proposition 1.2.2), we easily deduce the following. 
 \begin{corollary}
For every $h>0$, there exists $C_h>0$ such that for every 
$f\in L^2(\nu)$
$$
\big\|\sum_{k=1}^n \mu(k) f\circ \tau^k\big\|_{2}\le \frac{C_hn}{(\log n)^h}
\|f\|_2\, .
$$
\end{corollary}
 Notice that, trivially, for $f\in L^p(\nu)$, $1\le p\le \infty$, we have $\|\sum_{k=1}^n \mu(k) f\circ \tau^k\|_{p}\le n\|f\|_{p}$. Hence, performing  interpolation between $L^1(\nu)$ and $L^2(\nu)$ on the one hand and 
 between $L^2(\nu)$ and $L^\infty(\nu)$, on the other hand, we easily derive 
 the following.  
\begin{corollary}\label{corLpmobius}
For every $h>0$ and every $p>1$, there exists $C_{h,p}>0$ 
such that or every 
$f\in L^p(\nu)$
\begin{equation}\label{Lpmobius}
\|\sum_{k=1}^n \mu(k) f\circ \tau^k\|_{p}\le \frac{C_{h,p}n}{(\log n)^h}
\|f\|_p\, .
\end{equation}
\end{corollary}

It is mentionned by Sarnak \cite{Sarnak} that Bourgain's approach 
allows to prove that for every $f\in L^2(\nu)$, 
$\frac{1}n\sum_{k=1}^n \mu(k)f\circ \tau^k\underset{n\to \infty}
\longrightarrow 0$, $\nu$-a.s. 
 In view of \eqref{Lpmobius}, one could 
wonder whether we have a rate in this $\nu$-a.s. convergence. 
  We shall prove the following. 
\begin{proposition}\label{momax}
For every $h>0$ and every $p>1$ there exists $C_{h,p}>0$ such that 
for every $f\in L^p(\nu)$,
\begin{equation*}
\big\|\sup_{n\ge 1}\frac{|\sum_{k=1}^n \mu(k) f\circ \tau^k|}{n/(\log n)^h}
\big\|_p\le C_{h,p}\|f\|_p\, .
\end{equation*} 
In particular, for every $h>0$, $\frac{(\log n)^h}n\sum_{k=1}^n \mu(k)f\circ \tau^k\underset{n\to \infty}
\longrightarrow 0$.
\end{proposition}
\begin{proof}
Let $p>1$ and $h>0$. Let  
$0<\varepsilon <\frac{p-1}{p(1+h)}$. Let $h'>h+1/\varepsilon $. 
 Let $f\in L^p(\nu)$. Denote 
 $$M_n=M_{n,h}(f):= \frac{(\log n)^h}n\sum_{k=1}^n \mu(k)f\circ \tau^k . $$
 Denote also $u_n:= [{\rm e}^{n^\varepsilon}]$.
 By Corollary \ref{corLpmobius}, there exists $C_{h',p}$ such that, 
for every $n\ge 1$,
$$
\|M_n\|_p\le \frac{C_{h',p}}{(\log n)^{h'-h}}\|f\|_p\, .
$$
In particular, we see that 
\begin{gather*}
\|\sup_{n\ge 1} |M_{u_n}|\|_p^p 
\le \sum_{n\ge 1 } \| M_{u_n}\|_p^p \le C \|f\|_p^p\sum_{n\ge 1} \frac1{n^{p\varepsilon 
(h'-h)}}\, ,
\end{gather*}
and the latter series converges by our choice of $h'$. 
 Now let $n\ge 1$ and $u_n< m\le u_{n+1}$. Write $m=u_n+k$. We have, 
writing $\sum_{i=1}^m=\sum_{i=1}^{u_n}+  \sum_{i=u_{n+1}}^m$
$$
|M_m| \le |M_{u_n}| + \frac{Cn^{\varepsilon h}}{u_n}\sum_{i=u_n+1}^{u_{n+1}}
|f|\circ \tau^i
$$
Hence, 
\begin{eqnarray*}
\max_{u_n<m\le u_{m+1}} |M_m|&\le& 
|M_{u_n}| + \frac{Cn^{\varepsilon h}}{u_n}\sum_{i=u_n+1}^{u_{n+1}}
|f|\circ \tau^i \cr  &\le& \sup_{\ell \ge 1} |M_{u_\ell}| + C \Big( 
\sum_{\ell\ge 1} \big( \frac{\ell^{\varepsilon h}}{u_\ell}\sum_{i=u_\ell+1}^{u_{\ell+1}}|f|\circ \tau^i\big)^p\Big)^{1/p}\, .
\end{eqnarray*}
Now, using that $u_{\ell+1}-u_\ell=\mathcal O(u_\ell /\ell^{1-\varepsilon})$, we see 
that there exists $C>0$ such that 
\begin{gather*}
\Big\|\sum_{i=u_\ell+1}^{u_{\ell+1}}|f|\circ \tau^i\Big\|_p\le \frac{Cu_\ell\|f\|_p }{\ell^{1-\varepsilon}}\, .
\end{gather*}
Hence 
\begin{equation*}
\|\sup_{m\ge 1} |M_m|\|_p\le \|\sup_{\ell \ge 1} |M_{u_\ell}|\|_p 
+ C\|f\|_p \Big( \sum_{\ell\ge 1} \frac1{\ell^{p(1-\varepsilon(1+h))} }
\Big)^{1/p}\, 
\end{equation*}
and the desired result follows since $p(1-\varepsilon(1+h))>1$.
\end{proof}

 %%%%%%%%%%%%%%%%%%%%%%%%%%%%%%%%%%%
%%%%%%%%%%%%%%%%%%%%%%%%%%%%%%%%%%%
%\section{\bf Weighted ergodic theorems using  Dirichlet convolution.}
%\section{\bf Stability results using  Dirichlet convolution.}
%\section{\bf Results using  Dirichlet convolution.}
\section{\bf Ergodic stability of the   Dirichlet convolution.}\label{Dirconv}
Let us recall the following basic fact. Let $a(n)$ and $b(n)$ be 
two arithmetical functions with summatory functions  $A(x)=
\sum_{n\le x} a(n)$ and $B(x)=
\sum_{n\le x} b(n)$. Then 
$$\sum_{n\le x} a* b(n)= \sum_{n\le x} a (n) B(\frac{x}{n})=
\sum_{n\le x} b (n) A(\frac{x}{n})\, .
$$

Recall that a function $f\, : \, [0,+\infty)\to (0,+\infty)$ 
is slowly varying if for every $K>0$, 
$$
\lim_{x\to +\infty} \frac{f(Kx)}{f(x)}=0\, .
$$

We start with a lemma a bit in the spirit of Wintner's theorem \cite{Wi}  p.~180, and that should be known from specialists in number theory. 

\begin{lemma}\label{lemconv}
Let $a(n)$ be  a non-negative arithmetic function such that 
$A(x)\sim x^\alpha L(x)$  as $x\to \infty$, for some $\alpha>0$ and some positive non-decreasing 
slowly varying function $L$. Let $b(n)$ be an arithmetic function 
such that $\sum_{n\ge 1} \frac{|b(n)|}{n^\alpha} <\infty$. Then 
\begin{equation*}
\lim_{n\to \infty}\, \frac{1}{A(n)}\sum_{k=1}^n b*a(k)
= \sum_{m= 1}^\infty\frac{b(m)}{m^\alpha}\, . 
\end{equation*}
\end{lemma}
\begin{proof}Denote $c(k)=b* a(k)$. Let $M\ge 1$ be   an integer fixed for the moment. By assumption $x:= \sum_{m\ge 1}\frac{b(m)}{m^\alpha}$ 
is well defined. Denote also $x_M:=\sum_{m\ge M+1}\frac{b(m)}{m^\alpha}$.
\begin{gather*}
\Big|x- \sum_{1\le k\le n} c(k)/A(n)\Big| 
\le \sum_{1\le \ell \le M}  \Big|\frac{b(\ell)}{\ell ^\alpha}
- \frac{b(\ell)A(n/\ell)}{A(n)} \Big| + |x_M|+  \sum_{M< \ell \le n} |b(\ell)| A(n/\ell)/A(n) \, .
\end{gather*}
By assumption, for every $1\le \ell \le M$, 
$A(n/\ell)/A(n)\underset{n\to\infty}\longrightarrow 1/\ell^\alpha$. Since 
$L$ is positive and non-decreasing, there exists $C>0$ such that,  
\begin{eqnarray}\label{estAn}
A(x )&\le& C x^\alpha  L(x) \qquad \, \forall x\ge 1\, ,\cr
 A(n /\ell)&\le &C \frac{n^\alpha L(n)}{\ell^\alpha }\qquad 
\forall n\ge 1,\, \forall 1\le \ell \le n\, .\end{eqnarray}
%where we have used the monotonicity of $L$. 
 Hence $\sum_{M< \ell \le n} |b(\ell)| A(n/\ell)/A(n)\le C\frac{n^\alpha L(n)}{A(n)}
\sum_{M<\ell  \le n} |b(\ell)|/\ell^\alpha$, and so
\begin{gather*}
\limsup_{n \to \infty}\Big| \sum_{1\le k\le n} \frac{c(k)}{A(n)} -x\Big| 
\le |x_M|+ C \sum_{ \ell >M}  \frac{|b(\ell)|}{\ell^\alpha} .\end{gather*}
As the right-term tends to 0 when $M$ tends to infinity, this  proves the result.
\end{proof}

\begin{lemma}
Let $a(n)$ be  a non-negative arithmetic function such that 
$A(n)\sim n^\alpha / (\log n)^\beta$ for some $\alpha, \beta>0$. 
Let $b(n)$ be an arithmetic function 
such that $\sum_{n\ge 1} \frac{|b(n)|(\log n)^\beta}{n^\alpha} <\infty$. Then 
\begin{equation*}
\lim_{n\to \infty}\, \frac{1}{A(n)}\sum_{k=1}^n b*a(k)
= \sum_{m= 1}^\infty\frac{b(m)}{m^\alpha}\, . \end{equation*}
\end{lemma}
\begin{proof}We proceed as above, using the same notation. 
Let $M\ge 1$ be a an integer fixed for the moment. We have
\begin{eqnarray*}
 \Big|x- \sum_{1\le k\le n}  \frac{c(k)}{A(n)}\Big| 
  &\le& \sum_{1\le \ell \le M}  \Big|\frac{b(\ell)}{\ell ^\alpha}
- \frac{b(\ell)A(n/\ell)}{A(n)} \Big| + |x_M|
\cr & &+  
\sum_{M< \ell \le \sqrt n} |b(\ell)| A(n/\ell)/A(n) +
\sum_{\sqrt n < \ell \le n} |b(\ell)| A(n/\ell)/A(n)\, .
\end{eqnarray*}
Now, 
$$\sum_{M< \ell \le \sqrt n} |b(\ell)| A(n/\ell)/A(n) 
\le C \sum_{ \ell>M} |b(\ell)| /\ell^\alpha \, ,
$$
and 
$$
\sum_{\sqrt n < \ell \le n} |b(\ell)| A(n/\ell)/A(n)\le 
(\log n)^\beta \sum_{\ell >\sqrt n} |b(\ell)| /\ell^\alpha 
\underset{n\to\infty}\longrightarrow 0\, ,
$$
by a result analogue to the Kronecker lemma. Then we conclude as above.
\end{proof}

\medskip

\begin{proposition}\label{cgw}
Let $a(n)$ be  a non-negative arithmetic function such that 
$A(n)\sim n^\alpha L(n)$ for some $\alpha>0$ and some non-decreasing 
slowly varying function $L$. Let $b(n)$ be an arithmetic function 
such that $\sum_{n\ge 1} |b(n)|/n^\alpha <\infty$, $\sum_{n\ge 1} b(n)/n^\alpha \neq 0$  and $a*b(n)\ge 0$ for every $n\ge 1$. Let $p>1$.
\begin{itemize}
\item [$(i)$] Assume that $a(n)$ satisfies to the dominated ergodic theorem in $L^p$. Then, $a*b(n)$ satisfies to the dominated  ergodic theorem either. 
\item [$(ii)$]
If moreover, $a(n)$ satisfies to the pointwise ergodic theorem in $L^p$ then 
 $a*b(n)$ satisfies to the pointwise  ergodic theorem either. 
 \end{itemize}
\end{proposition}
\begin{remark} If $A(n)\sim n^\alpha /(\log n)^\beta$, for 
some $\beta>0$, then the conclusion of the theorem 
holds as soon as $\sum_{n\ge 1} |b(n)|(\log n)^\beta/n^\alpha <\infty$ 
and $\sum_{n\ge 1} b(n)/n^\alpha\neq 0$. When the 
pointwise ergodic 
theorem holds, the limit may be identified for the weigth $a*b(n)$ whenever 
it is identified for the weight $a(n)$. 
 \end{remark}
 \begin{proof} Let $(X, {\mathcal A},\nu,\tau)$ be a dynamical system.  Let $f\in L^p(\nu)$. By Lemma \ref{lemconv}, it suffices to prove  a maximal inequality and 
the almost-everyhere convergence for $\big(\frac{\sum_{1\le k\le n} c(k) f\circ \tau^k}{A(n)}\big)_{n\ge 1}$, where, as before, $c(n)=a*b(n)$.

\vskip 2pt

Let us prove $(i)$. Write 
$${\mathcal A}_\ell= {\mathcal A}_\ell(f)
=\sup_{n\ge 1} \frac{|\sum_{1\le k\le n} a(k)f\circ \tau^{\ell k}|}{A(n)}
\, . 
$$
By assumption, there exists $C>0$ (independent on $\ell$ anf $f$) such that 
\begin{equation}\label{hopf}
\|{\mathcal A}_\ell\|_{p,\nu}\le C \|f\|_{p,\nu}\, .
\end{equation}
 Using \eqref{estAn}, 
we see that 
\begin{eqnarray}
\nonumber\frac{|\sum_{1\le k\le n} c(k) f\circ \tau^k|}{A(n)}
&\le &\frac{\sum_{1\le \ell \le n} b(\ell)| \sum_{1\le k\le n/\ell} 
a(k) f\circ \tau^{\ell k}|}{A(n)} \\
\label{inemaxconv0} &\le &  \frac{\sum_{1\le \ell \le n} b(\ell)A(n/\ell) {\mathcal A}_\ell }{A(n)}
\le C \sum_{\ell\ge 1}\frac{|b(\ell)|}{\ell^\alpha} {\mathcal A}_\ell\, .
\end{eqnarray}
and we deduce the desired maximal inequality from \eqref{hopf} and the convergence of $\sum_{\ell\ge 1}\frac{|b(\ell)|}{\ell^\alpha}$.

\vskip 2pt

Let us prove $(ii)$. By assumption, there exist functions $(f_\ell)_{\ell\ge 
1}$, such that for every $\ell\ge 1$, $\big((\sum_{1\le k\le n} a(k)f\circ \tau^{\ell k})/A(n)\big)_n$ converges 
$\nu$-a.s. (and in $L^p(\nu)$) to $f_\ell$. Moreover, $\|f_\ell\|_p\le \|f\|_p$. Hence, 
$g:= \sum_{\ell \ge 1} \frac{b(\ell)}{\ell ^\alpha} f_\ell$ 
is well-defined in $L^p$ and $\nu$-a.s. Let us prove 
that 
$$
\sum_{1\le k\le n} c(k)f\circ \tau^k)/A(n) \underset{n \to\infty}
\longrightarrow g\qquad \mbox{$\nu$-a.s.}
$$

\medskip

 Let $M\ge 1$ be an integer, fixed for the moment. 
 We have 
\begin{eqnarray*}
\sum_{1\le k\le n} c(k) f\circ \tau^k
&=& \sum_{1\le \ell \le n} b(\ell) \sum_{1\le k\le n/\ell} 
a(k) f\circ \tau^{\ell k}\\ &=& \sum_{1\le \ell \le M} b(\ell) \sum_{1\le k\le n/\ell} 
a(k) f\circ \tau^{\ell k}+ \sum_{M< \ell \le n} b(\ell) \sum_{1\le k\le n/\ell} 
a(k) f\circ \tau^{\ell k}\, .
\end{eqnarray*}
 Let $g_M:= \sum_{\ell \ge M+1} \frac{b(\ell)}{\ell ^\alpha} f_\ell$. We have 
\begin{eqnarray*}
\Big| \sum_{1\le k\le n} c(k)f\circ \tau^k/A(n) -g\Big| 
&\le & \sum_{1\le \ell \le M}  \Big|\frac{b(\ell)}{\ell ^\alpha} f_\ell
- \frac{b(\ell)A(n/\ell)}{A(n)}  \frac1{A(n/\ell)} \sum_{1\le k\le n/\ell} 
a(k) f\circ \tau^{\ell k}\Big|\\ & &+ |g_M|+  \sum_{M< \ell \le n} b(\ell) \sum_{1\le k\le n/\ell} 
a(k) f\circ \tau^{\ell k}\, .
\end{eqnarray*}
%By assumption, for every $1\le \ell \le M$, 
%$A(n/\ell)/A(n)\underset{n\to\infty}\longrightarrow 1/\ell^\alpha$.
 Hence we infer that
\begin{equation}
\label{limsup}\limsup_{n\to \infty} \Big| \sum_{1\le k\le n} c(k)f\circ \tau^k)/A(n) -g\Big| 
\le |g_M|+ C \sum_{ \ell >M} \frac{b(\ell)}{\ell^\alpha} 
{\mathcal A}_\ell(f) \underset{M\to\infty}\longrightarrow 
0 \qquad \mbox{$\nu$-a.s.} \, ,
\end{equation}
and the result follows.
\end{proof}

Before giving examples, we would like to show that the previous result has a 
$L^{1,\infty}$ (weak-$L^1$) version. Recall that $f\in L^{1,\infty}$ if and only if 
$$
\|f\|_{1,\infty}:=\sup_{\lambda >0}\lambda\nu(\{x\in X\,:\, |f(x)|>\lambda \}) <\infty\, .
$$
The vector space $L^1_{1,\infty}$ equipped with $\|\cdot \|_{1,\infty}$ is not 
a normed-space, but we have the following estimate due to Stein and Weiss 
\cite[Lemma 2.3]{SW}. The form stated here is quoted from 
\cite[Lemma 4]{DQ}.

\begin{lemma}\label{SW}
Let $(g_n)_{n\in \N}$ be functions in $L^{1,\infty}(X,{\mathcal A},\nu)$. 
Assume that 
$$\sum_{n\in \N}\|g_n\|_{1,\infty} \log^+(1/\|g_n\|_{1,\infty})
<\infty. $$
Then the series $\sum_{n\in \N}g_n$ converges  $\nu$-a.s. to an 
element of $L^{1,\infty}(X,{\mathcal A},\nu)$. Moreover, writing $L:= \sum_{n\in \N}\|g_n\|_{1,\infty}$ 
and $K:= \sum_{n\in \N}\frac{\|g_n\|_{1,\infty}}{L} \log(L/\|g_n\|_{1,\infty})$, 
we have,
$$
\Big\|\sum_{n\in\N} g_n\Big\|_{1,\infty} \le 2(K+2)L\, .
$$
\end{lemma}

We say that $(w_k)_{k\ge 1}$ is a {\it good weight for the dominated ergodic theorem in 
$L^{1,\infty}$},  if there exists $C>0$ such that for every (ergodic) 
dynamical system $(X, {\mathcal A},\nu,\tau)$ and every $f$ in $L^p$, 
$$
\Big\|\sup_{n\ge 1}\frac{|\sum_{1\le k\le n} w_kf\circ \tau^k|}{W_n}
\Big\|_{1,\infty}\le C_p\|f\|_{1,\infty}\, .\
$$

\begin{proposition}\label{cgw2}
Let $a(n)$ be  a non-negative arithmetic function such that 
$A(n)\sim n^\alpha L(n)$ for some $\alpha>0$ and some non-decreasing 
slowly varying function $L$. Let $b(n)$ be an arithmetic function 
such that $\sum_{n\ge 1} |b(n)|/n^\alpha \log^+(n^\alpha/b(n))<\infty$, $\sum_{n\ge 1} b(n)/n^\alpha \neq 0$  and $a*b(n)\ge 0$ for every $n\ge 1$. 
\begin{itemize}
\item [$(i)$] Assume that $a(n)$ satisfies to the dominated ergodic theorem in $L^{1,\infty}$. Then, $a*b(n)$ satisfies to the dominated  ergodic theorem either. 
\item [$(ii)$]
If moreover, $a(n)$ satisfies to the pointwise ergodic theorem in $L^{1,\infty}$ then 
 $a*b(n)$ satisfies to the pointwise  ergodic theorem either. 
 \end{itemize}
\end{proposition}
\begin{proof}
The proof of the maximal inequality follows from \eqref{inemaxconv0} 
and Lemma \ref{SW}. Let us prove the pointwise ergodic theorem. As in the proof of Proposition \ref{cgw}, \eqref{limsup} holds true. Now, 
the sequence $(g_M)_{M\ge 1}$ from the proof (part $(ii)$) of Proposition \ref{cgw},   converges  $\nu$-almost surely to 0. Moreover, the non-increasing sequence 
$\sum_{\ell >M}\frac{b(\ell)}{\ell^\alpha}{\mathcal A}_\ell(f))_{M\ge 1}$ 
converges $\nu$-a.s. and its limit must be 0, since, by Lemma \ref{SW} 
it converges in probability to 0. 
\end{proof}

  \begin{theorem}  \label{listex} The arithmetical functions  
 \begin{eqnarray}
 \label{listext} \qq \begin{cases}  \s_s(k) &\qq\hbox{the sum of $s$-powers of divisors of $k$, $s\not=0$,}
 \cr 
  \theta(k) &\qq\hbox{the number of squarefree  divisors of $k$,}
 \cr 
  J_s(k) &\qq\hbox{the generalized Euler totient function, $s>0$,}
 \cr 
  |\m(k)| &\qq\hbox{where $\m$ is the M\"obius function,}
\end{cases}
\end{eqnarray}   are  good weighting functions  
 for the dominated ergodic theorem in 
$L^p$, $p>1$, and good weighting  functions for the pointwise ergodic theorem in $L^p$, $p> 1$. Moreover, $\sigma_s$ ($s\neq 0$), $|\mu|$ and $J_s$ ($s>0$) 
are good weighting functions for the dominated ergodic theorem in $L^{1,\infty}$ 
and for the pointwise ergodic theorem in $L^{1,\infty}$.  \end{theorem}
 \begin{proof} (i)   Denote for $s\in \R$ and all integers $n$,   $\varsigma_s(n)= n^s$ and let $\mathbb{ I}= \varsigma_0$.   We have  $\s_s= \mathbb{ I}*\varsigma_s$. If $s<0$, using Birkhoff's Theorem, we see that Proposition \ref{cgw} applies well. Indeed take $a(n)=1$, $b(n)=n^s$, $\a=1$. Obviously,  $\sum_{n\ge 1}  b(n)n^{-1}  =\sum_{n\ge 1}   n^{-1-|s|} <\infty$ and $\sum_{n\ge 1}  b(n)n^{-1} \neq 0$. Thus $\s_s(n)$ are good weights for the pointwise ergodic theorem in $L^p$, $p\ge 1$ and good weights for the dominated ergodic theorem in $L^p$, $p> 1$. If $s>0$,  
%  by using Theorem 3.1-(i) in \cite{LW} for instance, we 
it is well-known (using Abel summation and Birkhoff ergodic theorem)  that for any $f\in L^p(\nu)$, $p\ge 1$, $\frac{1}{n^{1+s}}\sum_{k\le n} k^s f\circ \tau ^k f(x)$ converges almost everywhere as $n\to \infty$. We apply Proposition \ref{cgw} with $a(n)=n^s$, $b(n)=1$, $\a=1+s$. This shows that $\s_s(n)$ are good weights for the pointwise ergodic theorem in $L^p$, $p> 1$. They are also good weights for the dominated ergodic theorem in $L^p$, $p> 1$, since $\frac{1}{n^{1+s}}|\sum_{k\le n} k^s f\circ \tau ^k f(x)|
\le \frac{1}{n}\sum_{k\le n} |f\circ \tau ^k f(x)|$.
 \vskip 2 pt (ii) Let us now consider the arithmetical function $\theta$. Introduce the arithmetical functions
 \begin{equation*}  \d(n)= 
 \begin{cases}1 & \quad n=1,\cr 0 &\quad  {\rm unless.}
  \end{cases} \qq \qq \tilde \m(n)=\begin{cases}\m(d) & \quad n=d^2,\cr 0 &\quad  {\rm unless.}
  \end{cases}\end{equation*}
Recall the fundamental inversion formula $\d =\mathbb{ I}*\m$. Writing $n=qm^2$, where $q$ is the product of those prime factors of $n$ with odd exponents, we first notice that 
$$ \m(n)^2= \m(m^2)^2=\d(m)=   \mathbb{ I}*\m(m)= \sum_{d^2|n} \m(d)= \sum_{u|n} \tilde\m(u) =  \mathbb{ I}*\tilde\m(n)   $$
since $d|m$ if and only if $d^2|n$.
Now as $d=\mathbb{ I}*\mathbb{ I}$, 
$$\theta(n)= \sum_{d|n} |\m(d)|= \sum_{d|n}  \m(d)^2= \sum_{d|n}\mathbb{ I}*\tilde\m(d)
=\mathbb{ I}*\mathbb{ I}*\tilde\m(n) =d*\tilde\m(n).  $$
Moreover,  $\sum_{n\ge 1} \frac{|\tilde\m(n)|} {n}\le \sum_{n\ge 1} \frac{1} {n^2}<\infty$ and $\sum_{n\ge 1} \frac{\tilde\m(n)} {n}= \sum_{n
\ge 1} \frac{\mu(n)} {n^2}=\frac1{\zeta(2)}\neq 0$ . The conclusion thus follows from Proposition \ref{cgw} and   Theorem \ref{mt}.
\vskip 2 pt (iii) Recall that  
$J_s(n)= \varsigma_s *\m(n)=\sum_{d|n} d^s \m(\frac{n}{d})$. The proof is very similar to the one of the case $\sigma_s(n)= n^s$, $s>0$. We apply Proposition \ref{cgw} with $a(n)=n^s$, $b(n)=\m(n)$, $\a=1+s$, noticing that $\sum_{n\ge 1} \frac{\m(n)}{n^{1+s}}=\frac1{\zeta(1+s)}\neq0$.
\vskip 2 pt (iv) This follows from Birkhoff's theorem since   $\m(n)^2 =  \mathbb{ I}*\tilde\m(n)$ and $\sum_{n\ge 1} \frac{|\tilde\m(n)|} {n} <\infty$
\end{proof}

%%%%%%%%%%%%%%%%%%%%%%%%%%%% 
%\section{\bf Brief description of Bourgain's approach.}
\section{\bf Sketch of Bourgain's approach.} \label{sketchbo}
%%%%%%%%%%%%%%%%%%%%%%%%%%%%%%%%%%%%%%%%%%%%%%%%%%%%%%%%
 Before passing to the preparation of the proof of Theorem \ref{mt1}, it is necessary   to briefly
recall the essential  steps of Bourgain's method.  We refer  ourselves to \cite{B2}. The basic reduction (Calderon's transference principle) to   the shift
model
 $(\Z, S)$,  where
$S\underline{z}=\{z_{\ell +1}, \ell\in \Z\}$, $\underline{z}=\{z_{\ell },
\ell\in
\Z\}$ can be presented as follows.   Let $(X,\a, \m, \tau)$ be a measurable dynamical
system and let $1<p\le \infty$. Let $J,N$ be  positive integers    with $J\gg N$. Let $f\in L^p(\m)$, $x\in X$ and define $\p$ on $\Z$   by 
\begin{eqnarray*}\p(j) =\begin{cases} f(\tau^j x)  & \hbox{if $0\le j\le J$,}
\\
0  & \hbox{ otherwise.}
\end{cases}
\end{eqnarray*}
We note that 
$$A_n^\tau f(\tau^j x)={1\over W_n} \sum_{k=0}^{n-1}w_k(S^k \p  )(j), \qq \quad n\le N,\ 0\le j<J-N .$$
Hence
$$
\sum_{0\le j<J-N}\sup_{n=1}^N |A_n^\tau f(\tau^j x)|\le \sum_{0\le j<J-N} \sup_{n=1}^N \Big|{1\over W_n} \sum_{k=0}^{n-1}w_kS^k \p(j)\Big|.
$$
Assume that we have proved that
\begin{eqnarray}\label{maxsm}\Big\| \sup_{n\ge 1} \big|{1\over W_n} \sum_{k=0}^{n-1}w_kS^k g(j)\big|
 \Big\|_{\ell^p(\Z,dj)}\le C_p \| g \|_{\ell^p(\Z,dj)},  
 \end{eqnarray}
for any $g\in \ell^p(\Z)$. Taking $g=\p$ we deduce,
$$
\sum_{0\le j<J-N}\sup_{n=1}^N |A_n^\tau f(\tau^j x)|^p\le C_p^p\sum_{0\le j\le J} | f(\tau^j x)|^p.
$$
By integrating with respect to $\mu$, it follows that $$
\sum_{0\le j<J-N}  \Big\| \sup_{n=1}^N|A_n^\tau f\circ \tau^j |\,\Big\|_p^p\le  C_p^p \sum_{0\le j\le J}  \|  f\circ \tau^j\|_p^p.
$$
Since  $\tau$ is $\m$-preserving, this finally leads to
$$  \big\| \sup_{n\ge 1}  |A_n^\tau f |\, \big\|_p \le  C(p)     \| f \|_p .
$$

Consider the kernel  $K_n:\ell^p(\Z)\to \ell^p(\Z)$ defined by $$K_n= {1\over W_n} \sum_{k=0}^{n-1}w_k\d_{\{k\}}.$$
By Fourier inversion formula, the maximal inequality on the shift model 
 $$\big\| \sup_{n\in \N}|K_{n}*f|\big\|_p\le C\|f\|_p,$$
 is equivalent  to
 $$\Big\|\sup_{n\in \N}
 \big|\int_0^1\overline{ \widehat K}_n(t)  \widehat f (t)\e^{2i\pi jt}dt\big|\Big\|_{\ell^p(\Z, dj)}<\infty. $$
 The latter is 
obtained by first proving a maximal inequality relatively to another kernel   $L_n$, whose Fourier transform is close to that of  $K_n$,  by using Fourier analysis, and next establishing an approximation result of the type 
\begin{equation}\label{comp}
\|\widehat K_n-\widehat L_n\|_\infty \le \frac{C}{(\log n)^{b}} \qquad 
\forall n\ge 2 \, 
\end{equation}
where $b$ is some positive constant. In several situations (in particular, 
when $w_n=d_n$), in order to deduce the maximal inequality for $K_n$, there is no loss to assume that $f\ge 0$ and to restrict $n$ to dyadic values ($n=2^k$, $k\in \N$). The plain inequality
$$ \sup_{k\in \N}|f*K_{2^k}|\le \sup_{k\in \N}|f*L_{2^k}|+ \big( \sum_{k\in \N}|f*(K_{2^k}-L_{2^k})|^2\big)^{1/2} $$
implies since $\|f*(K_{2^k}-L_{2^k}\|_2\le \|K_{2^k}-L_{2^k}\|_\infty\|f\|_2$,
$$\big\| \sup_{k\in \N}|f*K_{2^k}|\big\|_2\le\big\| \sup_{k\in \N}|f*L_{2^k}|\big\|_2+ \big( \sum_{k\in \N}\|K_{2^k}-L_{2^k}\|_\infty^2\big)^{1/2} \|f\|_2.$$
 
 Now let $\rho>1$ and denote $I_\rho:= \{[\rho^n]\,:\, n\in \N\}$.
The convergence almost everywhere    will result from the inequality: for every $\rho >1$ and every sequence $(N_j)_{j\ge 1}$, 
with $N_{j+1}\ge 2N_j$, 
\begin{equation}\label{convae}
\sum_{1\le j\le J} \Big\|\sup_{N_j\le N\le N_{j+1}\atop N\in I_\rho}|A_nf -A_{N_j}f|\Big\|_2 \le o(J)\|f \|_2,
\end{equation}
for $J$ large depending on $\rho$.
Consequently, once the reduction to the shift model operated, the main steps in applying Bourgain's approach are summarized in (\ref{comp}) and (\ref{convae}), see   (\ref{unifest}) and Theorem \ref{ae}. The next sections are devoted to the  necessary preparatory
steps   for the application of this method.
%%%%%%%%%%%%%
\begin{remark}[Maximal shift inequality] 
%%%%%%%%%%%%
 \label{maxcesaro}For the Ces\'aro kernel $\k_n= {1\over  n} \sum_{k=0}^{n-1}
\d_{\{k\}}$, the maximal shift inequality  for $p>1$ writes (after variable change),
%\begin{eqnarray}\label{maxusuel}\sum_{j\in \Z} \sup_{n\ge 1} \Big|{1\over n} \sum_{k=j}%^{j+n-1} g(k )\Big|^p\le C^p_p \sum_{j\in \Z}|g(j)|^p,  
% \end{eqnarray} or after variable change 
\begin{eqnarray}\label{imaxcesaro}\sum_{i\in \Z} \sup_{j\ge i}  {1\over j- i+1} \Big|\sum_{l=i}^{j} g(l )\Big|^p\le C^p_p \sum_{i\in \Z}|g(i)|^p.  
 \end{eqnarray}It suffices to prove it for  $g\ge 0$. Assume first that ${\rm support}(g)= \Z_-$. Then the only sums playing a role are those with $i\le j\le -1$ and the left-term writes 
 $$\sum_{\d \ge 1} \sup_{1\le \g\le \d} \Big(\frac{1}{\d-\g+1}\sum_{u=\g}^\d g(-u)\Big)^p.$$ Applying  Hardy and Littlewood  
maximal inequality (\cite{HL}, Theorem 8),  
%   for $a_n\ge 0$,
 \begin{equation}\label{hl}\sum_{j=1}^\infty \max_{1\le i\le j}\Big(\frac{1}{j-i+1}\sum_{l=i}^{j} a_l\Big)^p< \big(\frac{p}{p-1}
\big)^p\sum_{n=1}^\infty a_n^p \qq \quad (a_n\ge 0),\end{equation}
shows that \eqref{imaxcesaro} is realized with $C_p=p/(p-1)$. Now if ${\rm support}(g)=(-\infty, M]$, we apply the previous estimate to $\widetilde g(k)= g(k+M)$ whose support is $\Z_-$. To pass  to the general case, we use monotone convergence theorem (letting $M$ tend to $+\infty$), which is justified since $g\ge 0$.\end{remark}

%%%%%%%%%%%%%%%%%%%%%%%%%%%%%%%%%%%%%%%%%%%%%%%%%%%%%%%%%%%%%%%%%%%%%%%%%%%%%%%%%%%%
\section{\bf Divisors estimates.}\label{divest}

Recall that the divisor function is defined by $
d(n):=  \#\{1\le d\le n :d|n\}.$
For every $x\in [0,1]$, define 
$$ D_n(x) :=  \sum_{1\le k\le n} d(k) \e^{2ik\pi x}. $$
Then\begin{eqnarray*}
%\label{Dirichlet}
  D_n(x)&= & \sum_{1\le k\ell \le n} \e^{2ik\ell \pi x}  
%\label{Dirichlet}
 =2 \sum_{1\le k\le \sqrt n} \sum_{1\le \ell\le  n/k}
\e^{2ik\ell \pi x} - \sum_{1\le k, \ell\le\sqrt n} \e^{2ik\ell \pi x} \\
&:=&\widetilde D _n(x)- \sum_{1\le k, \ell\le\sqrt n} \e^{2ik\ell \pi x} 
\end{eqnarray*}
 It is well-known that 
\begin{eqnarray}\label{dnest} D_n:=D_n(0)= n(\log n + 2 \gamma-1)+ O(n^{1/3}))\, ,
\end{eqnarray}
where $\gamma$ is the Euler constant. Better estimates of the error term 
exist, but we shall not need them.
 Several asymptotics for $(D_n(x))_n$ may be found 
in Jutila \cite{J} when $x$ is rational or in Wilton \cite{Wil} 
for general $x$ under conditions on the continuous fractions expansion of 
$x$. 

\medskip

We shall need quantitative asymptotics according to the fact that 
$x$ is close to rational numbers with small or large denominators. 
In particular, it is unclear how to derive the results that we need from the above mentionned papers. 

\medskip

Our estimates use very simple ideas and we do not make use of the Voronoi 
identity related to the problem. 
 Actually, we shall rather estimate $\widetilde D_n(x)$. We note throughout by $a\wedge b$ the greatest common divisor of the positive integers $a$ and $b$.

\medskip

\begin{lemma}\label{expdiv}~~
%\begin{itemize}
There exists $C>0$, such that for every $1\le a \le q$ with $a\wedge q=1$, or $a=0$, $q=1$,  
and every $n\ge 1$, we have 
\begin{equation}\label{ratest}
 |D_n(a/q)- \frac{n}q( \log n-2\log q+2\gamma-1)|\le C (\sqrt n + q)
 \log (q+1))
\, .
\end{equation}
%\item [$(ii)$] There exists 
%$C>0$ such that such that for every $1\le a\le q\le n$ with $a\wedge q=1$. 
%$$
%|D_n(a/q)|\le  C( \frac{n\log n }{q}+\sqrt n \log (q+1)+ q\log (q+1)) \, .
%$$
%\end{itemize}
\end{lemma} 
\begin{proof} The case $a=0$, $q=1$ follows from (\ref{dnest}).
\vskip 1pt 
\noindent {\bf 1.} Assume first that $q\le \sqrt n$. 
  We split the sum defining $\widetilde D_n$ according to the 
fact that $k$ is a multiple of $q$ or not. We use the following obvious facts. 

-- If $q|k$, we have 
$$\sum_{1\le \ell\le  n/k}
\e^{2ik\ell \pi a/q}=[  n/k].
$$ 

-- If there exists $1\le s\le q-1$, such that $k\equiv s$ mod $q$, we have 
$$
|\sum_{1\le \ell\le  n/k} \e^{2ik\ell \pi a/q}|\le 
\frac{2}{|1- \e^{2i s \pi a/q}|}
$$
 Now, there are $[\sqrt n/q]$ multiple of $q$ less than $\sqrt n$ and for 
every $1\le s\le q-1$, there are at most $[\sqrt n/q]$ integers 
smaller than $\sqrt n$ and congruent to $s$ mod $q$. 

\vskip 2 pt

Notice that $s\to a s $ is a bijection of $\Z/q\Z-\{0\}$ and that 
there exists $C>0$, such that for every $1\le s'\le q-1$, 
$$
\frac{2}{|1- \e^{2i \pi  s'/q}|}\le \frac{ C q }{\min (s',q-s')}\, .
$$
 Hence, writing $\Gamma_n:=\{ 1\le k\le \sqrt n ~:k\notin q\Z\}$, 
$$
| \sum_{k\in \Gamma_n} \sum_{1\le \ell\le  n/k}
\e^{2ik\ell \pi a/q}|\le [\sqrt n/q] \sum_{1\le s\le q/2} \frac{Cq}{s}
\le \widetilde C \sqrt n \log (q+1)\, .
$$

Recall (see for instance Tenenbaum \cite{Tenenbaum} page 6) that there exists 
a universal constant $C>0$, such that for every $n\ge 1$, 
\begin{equation}\label{ten}
\big|\sum_{1\le m\le n}\frac1m - \log n -\gamma\big|\le \frac{C}{n}\,,
\end{equation}
where $\gamma$ is Euler's constant.

 \vskip 2pt

 Then, using that $|\log (\frac{\sqrt n}{q})-\log(\Big[\frac{\sqrt n}{q}
\Big])\le \frac{2q}{\sqrt n}$, we infer that, 

\begin{eqnarray*}
 \widetilde D_n(a/q)&=& 2\sum_{1\le m\le [\sqrt n/q]} n/(mq) + \mathcal O(\sqrt n\log(q+1)) 
 \\ &=&\frac{n}q( \log n-2\log q+2\gamma) + \mathcal O(\sqrt n)+ \mathcal O(\sqrt n\log(q+1))\, ,
\end{eqnarray*}
where the "big $ \mathcal O$" are uniform in the parameters. 

 \vskip 2pt

Similar computations  give, 
$$
\sum_{1\le k, \ell\le\sqrt n} \e^{2i\pi k\ell \pi \frac{a}{q}} = \frac{n}q + O(\sqrt n 
\log(q+1))\, .
$$
 \vskip 2pt
\noindent {\bf 2. } Assume now that $q>\sqrt n$. We use a similar reasonning as above. 
  In that case no integer $k$,
$1\le k \le \sqrt n$, is a multiple of $q$ and $\{ak ~:~1\le k \le \sqrt n\}$ 
is a set of integers with distinct residues modulo $q$. 

\medskip

Hence, 
$$
|D_n(a/q)|\le \sum_{1\le k\le \sqrt n} \frac2{|1-\e^{2i\pi ka/q|}}
\le C\,  \sum_{1\le |\ell |\le q/2} \frac{2q}{|\ell|}
 \le C q\log (q+1)\, .
$$
Similarly,
$$
\Big| \sum_{1\le k, \ell\le\sqrt n} \e^{2ik\ell \pi \frac{a}q}
\Big|\le C q\log (q+1)\, .
$$ 
Now, since $q>\sqrt n$, we see that $\frac{n}q |\log n-2\log q+2\gamma-1|\le C
q \log (q+1)$, and the lemma is proved.
\end{proof}

\medskip

Now let $(P_n)_{n\ge 1}$ and $(Q_n)_{n\ge 1}$ be  non-decreasing sequences of integers, such that for every $n\ge 1$, $1\le P_n\le Q_n\le n$. 

\begin{lemma}\label{lemkern}
Let $1\le a \le q\le P_n$ with $a\wedge q=1$, or $a=0$, $q=1$.  Let $x\in [0,1]$ be such 
that $|x-a/q|\le 1/Q_n $. There exists some universal 
constant $C>0$ such that, for every $n\ge 1$, 
\begin{gather}\label{kernelest1}
\Big| D_n(x) -\frac{1}q\sum_{1\le k\le n} \log k\, \e^{2ik\pi (x-a/q)}
-\frac{2(\gamma-1-\log q)}{q}\sum_{1\le k\le n}  \e ^{2i\pi k(x-a/q)}\Big|
\\
\nonumber \le C \Big( \frac{n^{3/2}\log n}{Q_n} +\frac{nP_n\log n}{Q_n}\, \Big)\, .
\end{gather}
In particular, there exists $\widetilde C>0$, such that, 
for every $n\ge 1$,
\begin{gather}\label{kernelest2}
\Big| D_n(x)-\frac{\log n}q\sum_{1\le k\le n} \e^{2ik\pi (x-a/q)} 
\Big| \le \widetilde C\Big(n  + \frac{n^{3/2}\log n}{Q_n}
+ \frac{nP_n\log n}{Q_n}\, \Big) \, .
\end{gather}
\end{lemma}
\begin{remark} The simpler estimate \eqref{kernelest2} will 
allow us to prove the oscillation inequality in $L^2(\mu)$. If 
$K_n= \frac{1}{D_n}\sum_{1\le k\le n} d(k)\d_{\{k\}}$ and $k_n=\frac{\log n}{D_n}\sum_{1\le k\le n} \d_{\{k\}}$, it will provide (upon suitable choice of $P_n, Q_n$) the estimate
$$ \big| K_n(x)-\frac{1}{q}k_n(x-a/q)
\big| \le
 \frac{C}{\log n}.$$ It is also 
sufficient to prove the maximal inequality in $L^p(\mu)$ for 
$3/2<p\le 2$. However, \eqref{kernelest1} seems to be needed 
to prove the maximal inequality for $1<p\le 2$.
\end{remark}
\begin{proof} We have, writing $R_n:=\frac{n}q( \log n-2\log q+2\gamma-1)$ and 
$R_0=0$,
\begin{eqnarray*}
D_n(x)&=&\sum_{1\le k\le n} d(k) \e ^{2i\pi kx} \ = \
\sum_{1\le k\le n} d(k) \e ^{2i\pi ka/q} \e ^{2i\pi k(x-a/q)} \\
&= &\sum_{1\le k\le n} \big(d(k)\e ^{2i\pi ka/q}-(R_k-R_{k-1})\big) \e ^{ik(x-a/q)} 
\\ & &\qq + \sum_{1\le k\le n} (R_k-R_{k-1})\e ^{2i\pi k(x-a/q)} 
\\ &:=&T_n+U_n\, .
\end{eqnarray*}

\noi Notice that 
\begin{eqnarray*}
q(R_k-R_{k-1})&=& k\log k -(k-1)\log (k-1)-2\log q +2\gamma-1\\
&=&\log k -2\log q +2\gamma -2 +\mathcal O(1/k)\, .
\end{eqnarray*}

  Hence,  
  \begin{gather*}
q U_n=  \sum_{1\le k\le n} \log k \, \e ^{2i\pi k(x-a/q)} + 
2(\gamma-1-\log q)\sum_{1\le k\le n}  \e ^{2i\pi k(x-a/q)} +O(\log n)\, .
\end{gather*}

\noi To deal with $T_n$ we use Abel summation by part. Recall that by Lemma 
\ref{expdiv},  for every $1\le k\le n$, 
$|D_k(a/q)-R_k|\le C\sqrt k (\log k+\log (q+1))$. We have

\begin{gather*}
T_n = \sum_{1\le k\le n} \Big( (D_k(a/q)-R_k)-(D_{k-1}(a/q)-R_{k-1})\Big)\e ^{2i\pi k(x-a/q)} \\
= \sum_{1\le k\le n} (D_k(a/q)-R_k) \e ^{2i\pi k(x-a/q)} (1-\e ^{2i\pi 
(x-a/q)})+
(D_n(a/q)-R_n)
\e ^{2i\pi(n+1)(x-a/q)} .
\end{gather*}

Hence,
\begin{gather*}
|T_n|\le |D_n(a/q)-R_n|+ \frac{C}{Q_n}
 \sum_{1\le k\le n} |D_k(a/q)-R_k|
 \le C\frac{n^{3/2}\log n}{Q_n}\, .
\end{gather*}

Let us prove \eqref{kernelest2}. Clearly, it suffices to handle 
the first term in \eqref{kernelest1}. We have

\begin{gather*}
\Big| \sum_{1\le k\le n} \log k\,  \e ^{2i\pi k(x-a/q)}-\log n \sum_{1\le k\le n} \e ^{2i\pi k(x-a/q)}\Big|
 \le \sum_{1\le k\le n}  |\log (k/n)| \le n\int_0^1 |\log t| dt\, ,
\end{gather*}
which finishes the proof.
\end{proof}

\begin{lemma}\label{expdiv1}
Let $x\in [0,1]$ be such that for every $1\le q\le P_n$ and every 
$0\le a\le q$, $|x-a/q|>1/Q_n$. There exists some absolute constant $C>0$ such 
that 
\begin{equation*}
|D_n(x)|\ \le \ C\,\Big(\frac{n\log n}{P_n} +\sqrt n \log n+ Q_n\log n 
+\frac{n^2\log n}{P_nQ_n}\Big)\, .
\end{equation*}
\end{lemma}

\begin{proof} By the Dirichlet principle, there exists $1\le a\le q$ with $a\wedge q=1$, 
such that $|x-a/q|\le 1/(qQ_n)\le 1/Q_n$. By assumption, we must have 
$q>P_n$, hence we have 
$$
|x-a/q|\le \frac{1}{P_nQ_n}\, .
$$

Then, using that $|\e^{2ip kx}-\e^{2i\pi ka/q}|\le |1-\e^{2i\pi k(x-a/q)}|
\le 2\pi k|x-a/q|\le 2\pi k/(P_nQ_n)$, we infer that
\begin{equation*}
|D_n(x)-D_n(a/q)|\ \le \ \frac{2\pi }{P_nQ_n} \sum_{1\le k\le  n}
 kd(k) \le \widetilde C \, \frac{  n^2\log n}{P_nQ_n}\, .
\end{equation*}
To conclude, we use  Lemma \ref{expdiv}, noticing that 
$q\ge P_n$.
\end{proof}

%%%%%%%%%%%%%%%%%%%%%%%%%%%%%%%%%%%%%%%%%%%%%%%%%%%%%%%%%%%%%%%%%%%%%%%%%%%%%%%%%%%%
\section{Maximal inequalities in $\boldsymbol {\ell^p}$}\label{fourierineq}

In this section we recall some results of Fourier analysis that may be 
found in \cite{Wierdl}, see also \cite{B1} or \cite{B2} 
for related results. 

\medskip

 In all that section, we denote by  $\eta \, :\, \R \to [0,1]$  a (fixed) smooth function 
such that
\begin{equation}\label{eta}\eta(x)=\begin{cases} 1& \quad {\rm  if}\ x\in [-\frac{1}{4}, \frac{1}{4}]  \cr
0& \quad {\rm  if}\ x\in \R\backslash [-1/2,1/2]
\cr 
  \hbox{is $C^\infty$}\ &\quad  {\rm  on}\     [-1/2,1/2]\backslash [-\frac{1}{4}, \frac{1}{4}].\end{cases}\end{equation}
Further,  $(w_n)_{n\ge 1}$ will be a sequence of elements 
of $\ell^1(\Z)$ such that for every $p>1$, there exists $C_p(\omega)>0$ such that 
\begin{equation}\label{wn}
\|\sup_{n\ge 1}|w_n *g|\, \|_{\ell^p(\Z)}\le C_p(w) \|g\|_{\ell^p(\Z)}
\qquad \forall g\in \ell^p(\Z)\, .
\end{equation}

We follow here the approach of Wierdl \cite{Wierdl}. 
However, as it has been noticed very recently by Mirek and Trojan 
\cite{MT}, there is a small gap in Wierdl's argument (on should have $q^p$ instead of $q$ in the equation after $**$ page 331), hence we shall 
sketch some of the proofs.  Our first lemma is just 
equation (24) of Wierdl \cite{Wierdl}, which is independent from the gap.

\begin{lemma}\label{lemBW}
There exists $M>2$ (depending solely on $\eta$) such that for 
every $p>1$, there exists $C_p>0$ such that for every $Q>1$ and every $1\le d\le Q/M$ and  every $h\in \ell^p(\Z)$  we have 
\begin{equation}\label{eqlemBW}
\Big\| \sup_n\Big|  \int_{-1/2}^{1/2}\widehat w_n(x)\eta(Qx) 
\widehat h(x){\rm e}^{2i\pi dj x}\, dx\Big| \,
\Big\|_{\ell^p(\Z,dj)}\le \frac{C_pC_p(w)}{d^{1/p}}\|h\|_{\ell^p}\, ,
\end{equation}
\end{lemma}

 Our second lemma is the correct version of Lemma $3'$ of Wierdl \cite{Wierdl}. 
The term $d^{1-1/p}$ does not appear in Lemma $3'$. Since we shall apply 
Lemma \ref{lemlp} for $p$ close to 1, it will turn out that this extra term will not be disturbing. 

\begin{lemma}\label{lemlp}
There exists $M>2$ (depending solely on $\eta$) such that for every 
$p>1$, there exists $C_p>0$ such that 
 for every 
$Q>1$, every $g\in \ell^p(\Z)$,  
$$
\Big\| \sup_n\Big| \sum_{1\le m\le d} \int_{-1/2}^{1/2}\widehat w_n(x)\eta(Qx) 
\widehat g(m/d+x){\rm e}^{2i\pi j (m/d+x)}\, dx\Big| \,
\Big\|_{\ell^p(\Z,dj)}\le C_pC_p(w)d^{1-1/p}\|g\|_{\ell^p}\, ,
$$
whenever $1\le d \le Q/M$.
\end{lemma}
\begin{proof}
We proceed as in Wierdl \cite{Wierdl}. We first assume that $g$ has 
finite support, i.e. there exists $N>0$, such that $g(k)=0$ whenever 
$|k|> N$. We have $\widehat g(m/d+x)= \sum_{k=-N}^N g(k) {\rm e}^{2ik\pi(m/d+x)}$. Notice that 
$\sum_{1\le m\le d}{\rm e}^{2i\pi (k+j)m/d}=d$ if $d|(k+j)$ and 0 otherwise.  Hence, for  every $x\in [1/2,1/2]$ and every $j\in \Z$, 
writing $j=td+r$ with $1\le t\le d$, we have 
\begin{eqnarray*}
\sum_{1\le m\le d}\widehat g(m/d+x){\rm e}^{2i\pi j (m/d+x)}& = &
\sum_{k=-N}^N g(k){\rm e}^{2i\pi (k+j) x} 
\sum_{1\le m\le d}{\rm e}^{2i\pi (k+j)m/d}\\
  & = & d\sum_{s  \in \Z} g(s d-j){\rm e}^{2i\pi sdx}\\
    & = & d{\rm e}^{2i\pi td x}\sum_{s  \in \Z} g(s d-r){\rm e}^{2i\pi sdx}
\end{eqnarray*}
Define $\widehat h_{d,r}\in \ell^p(\Z)$ (with finite support) by its Fourier transform: $$\widehat h_{d,r}(x):= d\sum_{s  \in \Z} g(s d-r){\rm e}^{2i\pi sdx}\, .
$$ Then, using Lemma \ref{lemBW}, we infer that 
\begin{eqnarray*}
& &\Big\| \sup_n\Big| \sum_{1\le m\le d} \int_{-1/2}^{1/2}\widehat w_n(x)\eta(Qx) 
\widehat g(m/d+x){\rm e}^{2i\pi j (m/d+x)}\, dx\Big| \,
\Big\|_{\ell^p(\Z,dj)}^p\\
\cr &=&\sum_{1\le r\le d} \Big\| \sup_n\Big|  \int_{-1/2}^{1/2}\widehat w_n(x)\eta(Qx) 
\widehat h_{d,r}(x){\rm e}^{2i\pi dtx}\, dx\Big| \,
\Big\|_{\ell^p(\Z,dt)}^p
\cr &\le& \frac{C_pC_p(w)}d \sum_{1\le r\le d} \|h_{d,r}\|_{\ell^p(\Z)}^p\, .
\end{eqnarray*}
By construction, $\|h_{d,r}\|_{\ell^p(\Z)}^p=d^p \sum_{s  \in \Z} g(s d-r)^p$. Hence, $\sum_{1\le r\le d} \|h_{d,r}\|_{\ell^p(\Z)}^p=
d^p\|g\|_{\ell^p(\Z)}^p$, and the result follows. The case where 
$g$ has no finite support may be deduced by approximation. 
\end{proof}

\begin{lemma}\label{lemel2}
There exists $M>2$ (depending solely on $\eta$) and $C>0$ such that 
 for every 
$Q>1$, every $g\in \ell^2(\Z)$,  
$$
\Big\| \sup_n\Big| \sum_{1\le m\le d} \int_{-1/2}^{1/2}\widehat w_n(x)\eta(Qx) 
\widehat g(m/d+x){\rm e}^{2i\pi j (m/d+x)}\, dx\Big| \,
\Big\|_{\ell^2(\Z,dj)}\le CC_2(w)\|g\|_{\ell^2}\, ,
$$
whenever $1\le d \le Q/M$.
\end{lemma}
\begin{proof}We have 
\begin{eqnarray*}
\Delta&:=&\Big\| \sup_n\Big| \sum_{1\le m\le d} \int_{-1/2}^{1/2}\widehat w_n(x)\eta(Qx) 
\widehat g(m/d+x){\rm e}^{2i\pi j (m/d+x)}\, dx\Big| \,
\Big\|_{\ell^2(\Z,dj)}^2 \\ &=&
\sum_{j\in \Z}  \sup_n\Big| \sum_{1\le m\le d} \int_{-1/2}^{1/2}\widehat w_n(x)\eta(Qx) 
\widehat g(m/d+x){\rm e}^{2i\pi j (m/d+x)}\, dx\Big|^2.
\end{eqnarray*}
For $1\le r\le d$, define $g_r$ by 
$$\widehat g_r(x)=\sum_{1\le m\le d}\widehat g(m/d+x){\rm e}^{2i\pi r (m/d+x)} .$$ Splitting the previous series into $d$ 
series according with the residue class of $j$ mod $d$ we see that 
$$
\Delta= \sum_{1\le r\le d} \sum_{j\in \Z}  \sup_n\Big|  \int_{-1/2}^{1/2}\widehat w_n(x)\eta(Qx) 
\widehat g_r(x){\rm e}^{2i\pi j dx}\, dx\Big|^2
$$
Notice that $\eta(\frac{Q}{2}\cdot)\eta(Q\cdot)=\eta(Q\cdot)$. By Lemma \ref{lemBW} applied with $\widehat h(x)=\eta(\frac{Q}{2}x)\widehat g_r(x)$, we have, by Parseval 
$$
 \sum_{j\in \Z}  \sup_n\Big|  \int_{-1/2}^{1/2}\widehat w_n(x)\eta(Qx) 
\widehat g_r(x){\rm e}^{2i\pi j dx}\, dx\Big|^2\le 
\frac{C_2}{d} \int_{-1/2}^{1/2}|\eta(\frac{Q}{2}x)\widehat g_r(x)|^2dx\, .
$$
Now, 
\begin{eqnarray*}
|\eta(\frac{Q}{2}x)\widehat g_r(x)|^2&=&\eta^2(\frac{Q}{2}x)\sum_{1\le m,m'\le d}\widehat g(m/d+x)
\overline{\widehat g(m'/d+x)}{\rm e}^{2i\pi r (m-m')/d}\, .
\end{eqnarray*}
Hence, using that $\sum_{1\le r\le d}{\rm e}^{2i\pi r (m-m')/d}$ is equal 
to 0 if $m\neq m'$ and to $d$ if $m=m'$, we obtain that
\begin{eqnarray*}
\sum_{1\le r\le d} |\eta(\frac{Q}{2}x)\widehat g_r(x)|^2 &= &
\eta^2(\frac{Q}{2}x)\sum_{1\le m,m'\le d}\widehat g(m/d+x)
 \overline{\widehat g(m'/d+x)}\sum_{1\le r\le d}{\rm e}^{2i\pi r (m-m')/d}\\
& = & \eta^2(\frac{Q}{2}x)\sum_{1\le m\le d}|\widehat g(m/d+x)|^2\, .
\end{eqnarray*}
Then, we infer that,
\begin{eqnarray*}
\Delta&\le &C_2^2 \sum_{1\le m\le d} \int_{-1/2}^{1/2}(\eta(\frac{Q}{2}x))^2 |\widehat g(m/d+x)|^2\, dx\\ &=& C_2^2 \sum_{1\le m\le d} \int_{-1/2}^{1/2}
(\eta(\frac{Q}{2}(x-m/d)))^2 |\widehat g(x)|^2\, dx .
\end{eqnarray*}
But, if $M>2$, the functions $(\eta(\frac{Q}{2}(\cdot-m/d)))_{1\le d
\le Q/m}$ have disjoint supports. Hence, 
$$
\Delta \le C_2^2\|g\|_{\ell^2}^2\, ,
$$
and the proof is complete. 
\end{proof}

\medskip

For every $s\ge 0$, define $\eta_s$ by
\begin{equation}\label{eta1}\eta_s(x)=\eta(4^s Mx),
\end{equation}  where $M$ is 
 a constant   such 
that Lemma \ref{lemBW} and Lemma \ref{lemel2} apply. 

\medskip

\begin{corollary}\label{corWierdl}
 Let $p>1$. For every $\varepsilon>0$, there exists $C_p>0$ such that for every $s\ge 1$ and every $1\le q<4^s$ 
 and every $g\in \ell^p(\Z)$,
\begin{eqnarray*}
%\label{firstbound}
\Big\| \sup_n\Big| \sum_{1\le a\le q,a\wedge q=1} \int_{-1/2}^{1/2}\widehat w_n(x)\eta_s(x) 
\widehat g(a/q+x){\rm e}^{2i\pi j (a/q+x)}\, dx\Big| \,
\Big\|_{\ell^p(\Z,dj)} \le C_pC_p(w)q^{1+\varepsilon-1/p}
\|g\|_{\ell^p}.\end{eqnarray*}
If $p=2$, \begin{eqnarray*}
\Big\| \sup_n\Big| \sum_{1\le a\le q,a\wedge q=1} \int_{-1/2}^{1/2}\widehat w_n(x)\eta_s(x) 
\widehat g(a/q+x){\rm e}^{2i\pi j (a/q+x)}\, dx\Big| \,
\Big\|_{\ell^2(\Z,dj)} \nonumber\le C_2C_2(w)q^\varepsilon
\|g\|_{\ell^2}\, .
\end{eqnarray*}
\end{corollary}
\begin{proof} As  we have for any function  on $\R$, $\sum_{k=1}^n F(\frac{k}{n})=  \sum_{q|n} \sum_{1\le a\le q \atop a\wedge q=1} F(\frac{a}{q})$, it follows from 
 M\"obius inversion formula  that \begin{equation}\label{Mobius}
\sum_{1\le a\le q, a\wedge q=1} F(a/q)=\sum_{d|q} \mu(q/d) \sum_{1\le m\le d}
F(m/d).\end{equation}
 Let $1\le q<2^s$.  Recall that (see e.g. Tenenbaum \cite{Tenenbaum}
 p. 83) there exists $c>0$ such that 
 $$\sum_{d|q}|\mu(d)|=2^{\omega(q)}\le 
 2^{c\frac{\log q}{\log \log q}}=O(q^\varepsilon),$$  where $\omega(q)$ is the number of prime divisors of $q$.  
 
 \medskip
 
 We shall apply \eqref{Mobius} with $F(\frac{m}d) =
 \int_{-1/2}^{1/2}\widehat w_n(x)\eta(Qx) 
\widehat g(m/d+x){\rm e}^{2i\pi j (m/d+x)}\, dx$. We combine it with  Lemma
\ref{lemlp} with $Q=4^sM$ if $p\neq 2$ and Lemma \ref{lemel2} 
if $p=2$.
\end{proof}

\medskip

\medskip

We shall now deal  with families of sequences $(w_n)_{n\ge 1}$ rather 
than with a single sequence. In particular, 
$\big((w_{n,q})_{n\ge 1}\big)_{q\ge 1}$ will be a family of elements 
of $\ell^1(\Z)$, such that for every $p>1$, there exists $C_p>0$ such 
that for every integer $s\ge 1$, there exists $K_s$ such that,
% for  every integer $q$, with $2^{s-1}\le q<2^s$, 
\begin{equation}\label{wnq}
\|\sup_{n\ge 1}|w_{n,q} *g|\, \|_{\ell^p(\Z)}\le C_pK_s \|g\|_{\ell^p(\Z)}
\qquad \forall g\in \ell^p(\Z)\ , \forall 2^{s-1}\le q<2^s.
\end{equation}

\begin{corollary}\label{corWierdl2} 
There exists $C>0$ such that for every $s\ge 1$, every $g\in \ell^2(\Z)$ 
and every family $\big((w_{n,q})_{n\ge 1}\big)_{2^{s-1}\le q<2^s}$ 
of elements of $\ell^1(\Z)$ satisfying \eqref{wnq}, we have
\begin{eqnarray*}
%\label{firstboundL2}
 \sum_{2^{s-1}\le q< 2^s} \Big\| \sup_n\Big| \sum_{1\le a\le q,a\wedge q=1} \int_{-1/2}^{1/2}\widehat w_{n,q}(x)\eta_s(x) 
\widehat g(a/q+x){\rm e}^{2i\pi j (a/q+x)}\, dx\Big| \,
\Big\|_{\ell^2(\Z,dj)}\\   \le CK_s2^{(\varepsilon+1/2)s}\, \|g\|_{\ell^2}\, .
\end{eqnarray*}
%where $M_{s,2}=\max_{2^{s-1}\le q\le 2^s} C_2(w^{(q)})$.
\end{corollary}
%\noindent {\bf Remark.} Using a similar result to Lemma 1 of Bourgain 
%\cite{Bourgain-IHES}, one can remove $"s"$ in the righ-hand side 
%of \eqref{firstboundL2} (and get $2^{s/2}$) in the above inequality. 
%We shall not need that refinement here.

\begin{proof}
For every $s\ge 1$, $\eta_{s-1}\equiv 1$ on 
$[-1/(M4^s),1/(M4^s)]$, hence $\eta_{s-1}\eta_s=\eta_s$. Moreover  the functions 
$\{x\to \eta_{s-1}(x-a/q)\}_{2^{s-1}\le q<2^s,1\le a\le q,a\wedge q=1}$ 
have disjoint supports. 

\smallskip

Indeed, let $s\ge 1$, $2^{s-1}\le q<2^s$ and $1\le a\le q$. Let 
$x\in [0,1]$ be such that $\eta_s(x-a/q)>0$. Then, 
$|x-a/q|\le \frac{1}{ 2\cdot 4^{s} M}$ and if $2^{s-1}\le q'<2^s$ and 
$1\le a'\le q'$, we have 

\begin{equation}\label{disjoint}
|x-a'/q'|\ge |a/q-a'/q'|-|x-a/q|\ge \frac{1}{2\cdot 4^{s}}-\frac{1}{ 4^{s}M}
\ge \frac{1}{ 2\cdot 4^{s}M} \, .
\end{equation}
Hence $\eta_s(x-a'/q')=0$. 
 In particular, writing 
$$ \widehat g_q (x)= \sum_{1\le a\le q} \eta_{s-1}(x-\frac{a}q) \widehat g(x) ,$$ we 
see that 
$$
 \sum_{1\le a\le q,a\wedge q=1}w_{n,q}(x-\frac{a}q)\eta_s(x-\frac{a}q) 
\widehat g(x){\rm e}^{2i\pi j x}= \sum_{1\le a\le q,a\wedge q=1} w_{n,q}(x-\frac{a}q)\eta_s(x-\frac{a}q) 
\widehat g_q(x){\rm e}^{2i\pi j x}\, .
$$
%Denote $M_s:=\max_{2^{s-1}\le q\le 2^s} C(w^{(q)})$. 
Applying Corollary \ref{corWierdl} and using that $\|g_q\|_{\ell^2}
=\|\widehat g_q \|_2$, we infer that 
\begin{eqnarray*}
& &\sum_{2^{s-1}\le q< 2^s} \Big\| \sup_n\Big| \sum_{1\le a\le q,a\wedge q=1} \int_{-1/2}^{1/2}w_{n,q}(x)\eta_s(x) 
\widehat g(a/q+x){\rm e}^{2i\pi j (a/q+x)}\, dx\Big| \,
\Big\|_{\ell^2(\Z,dj)}\\ 
&\le &CK_s2^{\varepsilon s} \sum_{2^{s-1}\le q< 2^s}\|\widehat g_q\|_{2}\le C K_s 2^{(\varepsilon+1/2)s} \, 
(\sum_{2^{s-1}\le q< 2^s}\|\widehat g_q\|_{2}^2)^{1/2}\\
&=& C K_s2^{(\varepsilon+1/2)s} \, 
\Big(\sum_{2^{s-1}\le q< 2^s} \sum_{1\le a\le q,a\wedge q=1}
\int_0^1|\widehat g(x)|^2 \eta^2_{s-1}(x-a/q) dx \Big)^{1/2}\\ 
&\le & CK_s2^{(\varepsilon+1/2)s}\|g\|_{\ell^2}\, ,
\end{eqnarray*}
where we used the above mentionned disjointness. 
\end{proof}

\medskip

%Using interpolation and taking advantage of Corollary \ref{corWierdl2}, 
% we derive now a much better bound, in average, than in 
% Corollary \ref{corWierdl}.

\begin{corollary}\label{corWierdl3}
 Let $p>1$. For every $\delta > 1/p$, there exists $C_{p,\delta}>0$ such that for every $s\ge 1$ 
 and every $g\in \ell^p(\Z)$,
\begin{equation}\label{firstbound3}
 \Big\| \sum_{2^{s-1}\le q< 2^s} \sup_n\Big| \sum_{1\le p\le q,a\wedge q=1} \int_{-1/2}^{1/2}\widehat w_{n,q}(x)\eta_s(x) 
\widehat g(a/q+x){\rm e}^{2i\pi j (a/q+x)}\, dx\Big| \,
\Big\|_{\ell^p(\Z,dj)}\le C_{p,\delta} K_s2^{s \delta}\|g\|_{\ell^p}\, .
\end{equation}
\end{corollary}
\begin{remark} Notice that the sum is inside the norm that time.
\end{remark}
\begin{proof}
 Let $s\ge 1$. Consider the following sub-additive 
and bounded (by Corollaries \ref{corWierdl} and \ref{corWierdl2}) operators 
on $\ell^r(\Z)$, $1< r\le 2$: 
\begin{equation*}
{\mathbb L}_s(g)= \sum_{2^{s-1}\le q< 2^s} \sup_n\Big| \sum_{1\le a\le q,a\wedge q=1} \int_{-1/2}^{1/2}\widehat w_{n,q}(x)\eta_s(x) 
\widehat g(a/q+x){\rm e}^{2i\pi j (a/q+x)}\, dx\Big|\, .
\end{equation*}
Let $1<p <2$ and chose any $r\in (1,p)$. 
 Let $\lambda\in (0,1)$ be the unique 
real number such that $1/p=\lambda /r +(1-\lambda)/2$. 

\vskip 2 pt By the Marcinkiewicz 
interpolation theorem (see e.g. Zygmund
 \cite[Th 4.6, Ch. XII, Vol. II]{Z}), there exists $C_{p,r}>0$ such that 
 \begin{equation*}
 \|{\mathbb L}_s g\|_{\ell^p} \le C_{p,r} K_s2^{s(2+\varepsilon-1/r)\lambda} 2^{(1-\lambda)(\varepsilon+1/2)s}\, 
 \|g\|_{\ell^p}=  C_{p,r} 2^{s[1+\varepsilon-1/r)\lambda+(1-\lambda)
 \varepsilon]} 2^{(1+\lambda)s/2}\, 
 \|g\|_{\ell^p}\, .
 \end{equation*}
Taking $r$ close enough to $1$, we may assume that 
$1+\varepsilon-1/r)\lambda+(1-\lambda)
 \varepsilon\le 3\varepsilon$.
 
 \medskip 
Notice that $\lambda= \frac{r(2-p)}{p(2-r)}$ and that  $(1+\lambda)/2=
\frac{r+p-rp}{p(2-r)}\underset{r\to 1}\longrightarrow 1/p$. The result follows since $\varepsilon$ may be taken arbitrary small. 
\end{proof}
\medskip

%%%%%%%%%%%%%%%%%%%%%%%%%%%%%%%%%%%
%%%%%%%%%%%%%%%%%%%%%%%%%%%%%%%%%%%

\section{Approximation result} \label{app}

%%%%%%%%%%%%%%%%%%%%%%%%%%%%%%%%%%%
%%%%%%%%%%%%%%%%%%%%%%%%%%%%%%%%%%%

In this section, we explain how to derive from the estimates 
on exponential sums,   good approximation results with  suitable 
Fourier kernels to which we can apply the previous maximal inequalities. 

 \subsection{}We use the notation (\ref{eta}), (\ref{eta1}).
%We shall need a few notations. 
% We still denote $\eta$ the smooth function defined in the previous 
%section. 
% We will make use of a constant $M>2$ such 
%that Lemma \ref{lemBW} and Lemma \ref{lemel2} apply. 
% Then, we still use the notation $\eta_s(x)=\eta(4^s M x)$. 
 Let $0<\tau \le 1$ be a parameter to be chosen later. 
 Assume that we have a collection $(\psi_{n,q})_{n\ge 1,q\ge 0}$ of 
complex-valued $1$-periodic functions on $\mathbb{R}$ such that there exists 
 $C>0$ such that for every $x\in [-1/2,1/2]$
 \begin{eqnarray}\label{psi}
 |\psi_{n,q}(x)|\le \frac{C }{(q+1)^\tau} \min(1, \frac{1}{n|x|}) \, .
 %\\|\psi_{n,0}|\le C \min(1, \frac{1}{n|x|}) \, .   \qquad 
 % \mbox{if $q\ge 1$}
 \end{eqnarray}
 
%We have in mind here to take  

Let $(P_n)_{n\ge 1}$ and 
$(Q_n)_{ n\ge 1}$ be two non-decreasing sequences of integers. Assume that
 there exist $R,S>1$   such that 
for every $n\ge 2$, 
\begin{eqnarray} \label{PnQn}
P_n \ge R(\log n )^{S/\tau}  \qquad \mbox{and} \qquad 
  16M P_n^2 \le  Q_n \le \frac{n}{(\log n)^{S(1+1/\tau)}}\, .
\end{eqnarray}
 In particular, 
\begin{equation}\label{PnQn2}
Q_n\ge 16MP_n \ge 16M R(\log n )^{S/\tau}\ge (\log n )^{S/\tau}\, .
\end{equation}
 
\vskip 2 pt
Denote 

$$
\M (P_n,Q_n)=\M_n:= \{x\in [0,1]~:~ \exists\,  0\le a\le q\le P_n~:~
|x-a/q|\le 1/Q_n\}\, .
$$
\vskip 4 pt
\noi Notice that, because of \eqref{PnQn}, if $x\in \M_n$, $x\neq 0$, there exist  unique  numbers $a_n(x)$ and $q_n(x)$ with $1\le a_n\le q_n$ and $a_n\wedge 
q_n=1$ and such that $|x-a_n(x)/q_n(x)|\le 1/Q_n$. Let us also define $a_n(0)=0$ 
and $q_n(0)=1$.

\medskip

%Let $\eta \, :\, \R \to [0,1]$ be a smooth function such that 
%$\eta \equiv 1$ on $[-1/4,1/4]$ and $\eta\equiv 0$ on $ \R\backslash [-1/2,1/2]$. 

%\medskip

%Define also functions $(v_n)_{n\ge 1}$ on $[0,1]$ by 

%$$
%v_n(\beta) =\int_0^1 \e^{-2i \pi n\beta x} dx \, .
%$$

%\medskip

%Let $(k_n)_{n\ge 1}\subset \ell^1(\Z)$ be good for the Hardy-Littlewood maximal inequality, i.e. such that for every $p>1$ there exists 
%$C_p>0$ such that for every $g\in \ell^p(\Z)$, 
%$$
%\|\sup_{n\ge 1} |k_n * g|\|_{\ell^p(\Z)}\le C_p\|g\|_{\ell^p(\Z)}\, .
%$$ 
%Denote $v_n:= \widehat k_n$ and assume that there exists $C>0$ such that 
%\begin{equation}\label{dirkern}
%|v_n(x)|\le C \min (n, \frac1{|x|}) \qquad \forall x\in [-1/2,1/2],\, 
%\forall n\ge 1\, .
%\end{equation}
%We have in mind to take $k_n$ such that $v_n(x)=1/n \sum_{1\le k\le n} 
%\e^{2i\pi kx}$ or "regular" perturbation of it (as in Lemma 
%\ref{lemkern}).

 \vskip 2 pt

Finally, define  functions $\varphi_n$ on $[0,1]$ by 

\begin{equation}\label{varphi}
\varphi_n(x):= \psi_{n,0}(x)\eta_0(x) + \sum_{s=1}^\infty \sum_{2^{s-1}\le q 
<2^s}  \sum_{1\le a\le q,a\wedge q=1} \psi_{n,q}(x-a/q) \eta_s(x-a/q)\, .
\end{equation}
\vskip 2 pt

\noi Notice that for any fixed $s$, the functions 
$$x\mapsto \eta_s(x-a/q), \qq\qq 2^{s-1}\le q<2^s,\ 1\le a\le q,\ a\wedge q=1,$$
 have 
disjoint supports. Hence, by \eqref{psi} the series defining $(\varphi_n)_{n\ge 1}$ 
are uniformly convergent. 

\medskip

%Indeed, let $s\ge 1$, $2^{s-1}\le q<2^s$ and $1\le a\le q$. Let 
%$x\in [0,1]$ be such that $\eta_s(x-a/q)>0$. Then, 
%$|x-a/q|\le \frac{1}{ 2\cdot 4^{s} M}$ and if $2^{s-1}\le q'<2^s$ and 
%$1\le a'\le q'$, we have 

%\begin{equation}\label{disjoint}
%|x-a'/q'|\ge |a/q-a'/q'|-|x-a/q|\ge \frac{1}{2\cdot 4^{s}}-\frac{1}{ 4^{s}M}
%\ge \frac{1}{ 2\cdot 4^{s}M} \, .
%\end{equation}
%Hence $\eta_s(x-a'/q')=0$. 

\medskip

%In particular, we see that for every $s\ge 1$ and every $x\in[0,1]$, the 
%sum 
%$$\sum_{2^{s-1}\le q 
%<2^s} \frac{1}{\phi_n(p,q)} \sum_{1\le p\le q} v_n(x-p/q) \eta_s(x-p/q)$$
% contains 
%only one term, hence  $\varphi_n$ is well-defined, by \eqref{phi}.

%\medskip

We shall need the following technical lemma, which is essentially due to 
Bourgain.

\begin{lemma}\label{lembourgain}
Let $(T_n)_{n\ge 1}$ be complex-valued functions on $[0,1]$, such that 
there exists $C>0$ and $\gamma>0$ such that for every $n\ge 2$,
\begin{eqnarray}
\label{firstcond}|T_n(x)-\psi_{n,q_n}(x-a_n/q_n)|\le 
\frac{C}{(\log n) ^\gamma} \qquad \forall x\in \M_n \\
\label{secondcond}|T_n(x)|\le \frac{C}{(\log n)^\gamma} \qquad 
\forall x\in [0,1]\backslash \M_n 
\, .
\end{eqnarray}
Then, there exists $\widetilde C>0$ such that for every $n\ge 2$ and every 
$x\in [0,1]$, 
\begin{equation}\label{estimate}
|T_n(x)-\varphi_n(x)|\le \frac{\widetilde C}{(\log n) ^{\min (\gamma, S)} }\, .
\end{equation}
\end{lemma}

\begin{proof}
\noindent{\bf 1.}  We start with the case where $x\in [0,1]\backslash \M_n$. 
By \eqref{secondcond}, 
it suffices to estimate $|\varphi_n|$.
By assumption, $\min(x,1-x)\ge \frac{1}{Q_n}$. Hence, 
 $$|\psi_{n,0}(x)|
\le \frac{C Q_n}{n} \le \frac{C}{(\log n)^{S(1+1/\tau)}}. $$
By Dirichlet's principle, there exists $1\le a\le q \le Q_n$, with 
$a\wedge q=1$ such that 
$$
|x-\frac{a}q|\le \frac{1}{qQ_n}\, .
$$
Let $1\le a'\le q'$ with $a'\wedge q'=1$ and $a'/q'\neq a/q$. Then, 
if $q'\le \frac{(\log n)^{S/\tau}}{2}$, using \eqref{PnQn} and 
\eqref{PnQn2}, we have
\begin{eqnarray*}
|x-\frac{a'}{q'}|\ge \frac1{qq'}-\frac1{qQ_n} = \frac{1}{q}( \frac1{q'}-\frac1{Q_n})
\ge \frac{1}{q(\log n)^{S/\tau}} \ge \frac{1}{Q_n(\log n)^{S/\tau} }
\ge \frac{(\log n)^S}n\, .
\end{eqnarray*}
Hence, when $q'\le  \frac{(\log n)^{S/\tau}}{2}$, 
$$|\psi_{n,q'}(x
-\frac{a'}{q'})|\le 
\frac{C}{(q'+1)^\tau(\log n)^S}. $$So, using 
\eqref{varphi}, we obtain 

$$
|\varphi_n(x)|\le \frac{1}{(\log n)^{S(1+1/\tau)}} + |\psi_{n,q}(x)|+ 
\frac{C}{(\log n)^S} \sum_{s\, :\, 2^{s}\le (\log n)^{S/\tau} }2^{- s\tau} + 
C \sum_{s\, :\, 2^{s}\ge (\log n)^{S/\tau} }2^{- s\tau} \, .
$$
Now, since $x\in [0,1]\backslash \M_n$, $q\ge P_n$ and $|\psi_{n,q}(x)|
\le C/(q+1)^\tau=O((\log n)^S)$. Hence the lemma 
is proved in that case. 

\medskip

\noindent {\bf 2.} Assume now that $x\in \M_n$. Suppose  $x\neq 0$. 
By assumption, 
$|x-a_n(x)/q_n(x)|\le 1/Q_n$ and $q_n(x)\le P_n$. Hence, if
$s\ge 1$, is such that $2^{s-1}\le q_n(x)< 2^{s}$, we have 
\begin{equation}
|x-a_n(x)/q_n(x)|\le 1/Q_n \le \frac{8MP_n^2}{Q_n} . \frac{1}{8MP_n^2} 
\le \frac{1}{2M4^{s}} \, .
\end{equation}
In particular, $\eta_s(x-a_n(x)/q_n(x))=1$. If $x=0$, $\eta_0(0)=1$.

\medskip

Let $1\le a'\le q'$ with $a'\wedge q'=1$ and $a'/q'\neq a/q$. Then, 
if $q'\le P_n$, using \eqref{PnQn}
\begin{eqnarray*}
|x-\frac{a'}{q'}|\ge \frac1{q_n(x)q'}-\frac1{Q_n} \ge \frac{1}{P_n ^2}-
\frac1{Q_n} \ge \frac{8M-1}{Q_n} \, ,
\end{eqnarray*}
and $|\psi_{n,q'}(x-\frac{a'}{q'})|\le \frac{8M-1}{(q'+1)^\tau
(\log n)^{S(1+1/\tau)}}\, ,$
 by \eqref{PnQn}.

\medskip

Finally, we obtain 
\begin{eqnarray*}
|\varphi_n(x)-T_n(x)|\le \frac{C}{(\log n)^\gamma}+ \frac{8M-1}{(\log n)^{1+1/\tau}} \sum_{s\, :\, 2^{s}\le P_n }2^{- \tau s} + 
C \sum_{s\, :\, 2^{s}> P_n }2^{- \tau s}\, ,
\end{eqnarray*}
which proves the lemma in that case. 
\end{proof}

\medskip
 \subsection{}
  Let us assume from now that there exists  a sequence $(w_{n,q})_{n\ge 1,q\ge 0}$ of 
 elements of $\ell^1$ such that assumption \eqref{psi} is satisfied with the choice $$\psi_{n,q}= \widehat w_{n,q} \qq\qq n\ge 1,q\ge 0.$$
 Introduce  the following assumption.
\vskip 2pt
  {\it For every $p>1$, there  exists $C_p>0$ such that} 
\begin{equation}\label{inemaxconv}
\|\sup_{n\ge 1} |w_{n,q}*g|\|_{\ell^p}\le \frac{C_p}{q^\tau}
\|g\|_{\ell^p}\qquad \forall g\in \ell^p\, .
\end{equation}

\begin{proposition} \label{maxshift}
Let $(K_n)_{n\ge 1}\subset \ell^1$, with $\sup_{n\ge 1}\|K_n\|_{\ell^1}
<\infty$.  Assume that $T_n:=\widehat K_n$ satisfies 
\eqref{firstcond} and \eqref{secondcond}, for some $\gamma>1/2$.
%Assume that for some $\gamma>1/2$,\begin{eqnarray}
%\label{firstconda}|\widehat K_n(x)-\widehat w_{n,q_n}(x-a_n/q_n)|\le 
%\frac{C}{(\log n) ^\gamma} \qquad \forall x\in \M_n \\
%\label{secondconda}|\widehat K_n(x)|\le \frac{C}{(\log n)^\gamma} \qquad 
%\forall x\in [0,1]\backslash \M_n \, .\end{eqnarray} 
  Assume moreover that 
\eqref{inemaxconv} holds. Then, for every $p\in (\frac1\tau +\frac{2-1/\tau}{2\min(\gamma,S)},2]$, there exists $C_p>0$, such that 
\begin{equation*}
\|\sup_{n\ge 1} |K_{2^n}*g|\|_{\ell^p}\le C_p
\|g\|_{\ell^p}\qquad \forall g\in \ell^p\, .
\end{equation*}
\end{proposition}
%ANCIEN ENONCE
%\begin{proposition} \label{maxshift}
%Let $(K_n)_{n\ge 1}\subset \ell^1$, with $\sup_{n\ge 1}\|K_n\|_{\ell^1}
%<\infty$. Assume that $T_n:=\widehat K_n$ satisfies 
%\eqref{firstcond} and \eqref{secondcond}, for some $\gamma>1/2$. Assume moreover that 
%\eqref{inemaxconv} holds. Then, for every $p\in (\frac1\tau +\frac{2-1/\tau}%{2\min(\gamma,S)},2]$, there exists $C_p>0$, such that 
%\begin{equation*}
%\|\sup_{n\ge 1} |K_{2^n}*g|\|_{\ell^p}\le C_p
%\|g\|_{\ell^p}\qquad \forall g\in \ell^p\, .
%\end{equation*}
%\end{proposition}
  \begin{remark} According to Section \ref{sketchbo}, Proposition \ref{maxshift} provides the maximal inequality   for the kernel $K_n$, and thereby in any measurable dynamical system. 
   \end{remark}
  
 For the proof, we will need the following Lemma. Let $L_n$ be the inverse Fourier  
transform of $\varphi_n$, which is made possible because of the introduction of the smooth function $\eta$. 
   \begin{lemma}\label{inemaxshift}
For every $p>1/\tau$, there exists $C_p>0$ such that for every 
$g\in \ell^p$, 
\begin{equation}
\|\sup_{n\ge 1} |g* L_n|\|_{\ell^p} \le C_p \|g\|_{\ell^p}\, .
\end{equation} 
\end{lemma}
%\noindent {\bf Remark.} The result should hold for \emph{any} $p>1$, by using 
%the proof of Bourgain along the squares.\\
\begin{proof}[Proof of Lemma \ref{inemaxshift}] 
Let $r>1/\tau$. We apply Corollary \ref{corWierdl3} (as $\psi_{n,q}=\widehat w_{n,q}$) to obtain that 
for every $\delta>1/r$ and every  $g\in \ell^r$,
\begin{eqnarray}
\nonumber \Big\| \sum_{2^{s-1}\le q< 2^s} \sup_n\Big| 
 \sum_{1\le p\le q,a\wedge q=1} \int_{-1/2}^{1/2}\widehat w_{n,q}(x)\eta_s(x) 
\widehat g(a/q+x){\rm e}^{2i\pi j (a/q+x)}\, dx\Big| \,
\Big\|_{\ell^r(\Z,dj)}\\ 
\label{estim0} \le C_{r,\delta} 2^{s (\delta-\tau)}\|g\|_{\ell^r}\, . 
\end{eqnarray}
We may chose $\delta <\tau$, so that $\sum_{s\ge 0} 2^{s(\delta-\tau)}
<\infty$. 
 Summing the estimates \eqref{estim0}  over $s\ge 1$  we infer that 
for every $g\in \ell^r$,  
$$
\|\sup_{n\ge 1} \int_{-1/2}^{1/2} 
\widehat L_n(x) \widehat g(x){\rm e}^{2i\pi j x} \, dx\|_{\ell^r(\Z)}
\le C_r \|g\|_{\ell^r(\Z)}\, .
$$
 Taking inverse Fourier transform we see that Lemma \ref{inemaxshift} 
is true.
\end{proof}

\begin{proof}[Proof of Proposition \ref{maxshift}]  
By Lemma \ref{lembourgain}, since \eqref{firstcond} and \eqref{secondcond} are satisfied,
we see that \eqref{estimate} holds. Hence, we have 
\begin{equation}\label{unifest}
\|\widehat K_n-\widehat L_n\|_\infty \le \frac{C}{(\log n)^{\min(\gamma,S)}} \qquad 
\forall n\ge 2 \, .
\end{equation}
Then, we infer that  for every $f\in \ell^2(\Z)$, 
\begin{equation}\label{L2}
\|f*(K_n-L_n)\|_{\ell^2(\Z)}\le \frac{C}{(\log n)^{\min(\gamma,S)}} 
\|f\|_{\ell^2(\Z)}\qquad \forall n\ge 2 \, .
\end{equation}

\medskip
Let $2>p>1/\tau$, be fixed for the moment. 
Since $\sup_{n\ge 1}\|K_n\|_{\ell^1(\Z)}\le C$, we see that 
for every $n\ge 1$ and for every $r\ge 1$ and $g\in\ell^r(\Z)$ (using Young's inequalities),  
$\|K_n* g\|_{\ell^r(\Z)}\le C \|g\|_{\ell^r(\Z)}.$ Hence, by \eqref{L2} and 
Lemma \ref{inemaxshift}, we see that, for every $n\ge 0$ and every $r>1
/\tau$,
\begin{eqnarray*}
\|K_{2^n}*g-L_{2^n}*g\|_{\ell ^2(\Z)} &\le& \frac{C}{n^{\min(\gamma,S)}}\|g\|_{\ell^2(\Z)}
\qquad \forall 
g\in \ell^2(\Z),\\
\|K_{2^n}*g-L_{2^n}*g\|_{\ell ^r(\Z)} &\le& C_r \|g\|_{\ell^r(\Z)}\qquad\qquad\quad \forall 
g\in \ell^2(\Z)
\end{eqnarray*} 

\smallskip
Let $1/\tau<r<p$.
 Interpolating, we deduce that  there exists 
$C_{p,r}$ such that for every $n\ge 0$, 
\begin{eqnarray*}
\|K_{2^n}*g-L_{2^n}*g\|_{\ell ^p(\Z)} \le \frac{C_{r,p}}{n^\sigma} 
\|g\|_{\ell^r(\Z)}\qquad \forall g\in \ell^p(\Z)\, ,
\end{eqnarray*} 
with $\sigma= \frac{2\tilde \gamma (p-r)}{p(2-r)}$ and 
$\tilde \gamma=\min(\gamma,S)$.
\medskip

Notice that 
 \begin{equation*}
\sigma -\frac1p =\frac{1}p \, \Big(\frac{2\tilde \gamma(p-r)}{2-r}-1 
\Big)
\underset{r\to 1/\tau}\longrightarrow
\frac{2\tilde \gamma(p-1/\tau)+1/\tau-2}{p(2-1/\tau)}\, .
\end{equation*}
Since $p>\frac1\tau +\frac{2-1/\tau}{2\min(\gamma,S)}$, we may chose $r$ close enough to $1/\tau$, such that $\sigma>1/p$. In particular for that choice, 
 $$
\|\sup_{n\ge 0} |K_{2^n}*g-L_{2^n}*g|\, \|_{\ell ^p(\Z)} 
\le \Big(\sum_{n\ge 0} \|K_{2^n}*g-L_{2^n}*g\|_{\ell ^p(\Z)}^p\Big)^{1/p}
\le C_{p,r} \|g\|_{\ell^p(\Z)}\qquad \forall g\in \ell^p(\Z)\, .
$$
\end{proof}

\medskip
 \subsection{} We now establish an estimate of the type (\ref{convae}) in order to prove the convergence almost everywhere. 
Recall that for $\rho>1$, we have noted $I_\rho= \{[\rho^n]\,:\, n\in \N\}$. Introduce the following assumption:
\vskip 2pt 
{\it For every $\rho >1$ and every sequence $(N_j)_{j\ge 1}$, 
with $N_{j+1}\ge 2N_j$, there exists $C>0$ such that, 
%writing $I_\rho:= \{[\rho^n]\,:\, n\in \N\}$},
\begin{equation}\label{oscconv}
\sum_{j\ge 1}\| \, \sup_{ N_j\le N\le N_{j+1}\atop N\in I_\rho}|(w_{N,q}-w_{N_j,q})*g|\, 
\|_{\ell^2(\Z)}^2\, 
\le \frac{C}{q^\tau} \|g\|_2 \qquad \forall g\in \ell^2\, .
\end{equation}
\begin{theorem}\label{ae}
Let $(K_n)_{n\ge 1}\subset \ell^1$. Assume that $T_n:=\widehat K_n$ satisfies 
\eqref{firstcond} and \eqref{secondcond}, for some $\gamma>1/2$. Assume moreover that 
\eqref{oscconv} holds. Then, for every $\rho >1$ and every sequence $(N_j)_{j\ge 1}$, 
with $N_{j+1}\ge 2N_j$, 
\begin{equation}
 \sum_{j= 1}^J\| \, \sup_{ N_j\le N\le N_{j+1},\, N\in I_\rho}|(K_{N}-K_{N_j})*g |
 \, \|_{\ell^2(\Z)}^2=o(J)\,  .
\end{equation}
\end{theorem}
 \begin{remark} According to Section \ref{sketchbo}, the convergence almost everywhere now follows from Theorem \ref{ae}.  
   \end{remark}
\begin{proof}
 The proof follows closely the argument p. 220 in Bourgain \cite{B}. 
Let $\rho>1$.  
 Let $(N_j)_{j\in \N}\subset I_\rho$ be an increasing sequence with 
$N_{j+1}>2N_j$. For every $j\in \N$, define a maximal operator by 
\begin{equation*}
M_jf=M_{j,\rho}f:= \sup_{N_j\le N< N_{j+1}, \, N\in I_\rho}
|f* K_{N}-f*K_{N_j}|\qquad  
\mbox{$\forall f\in \ell^2(\Z)$}.
\end{equation*}

As in the previous proof we define $L_n$ as the inverse Fourier transform 
of $ \varphi_n$. 
 Notice that, for every $f\in \ell^2(\Z)$,

\begin{eqnarray*}
M_j&\le& \sup_{N_j\le N< N_{j+1}, \, N\in I_\rho}
|f* L_{N}-f*L_{N_j}|+ 2 \sup_{N_j\le N< N_{j+1}, \, N\in I_\rho}
|f* (L_{N}-K_N)|\\
&: =& \widetilde M_j + 2 \sup_{N_j\le N< N_{j+1}, \, N\in I_\rho}
|f* (L_{N}-K_N)|\, .
\end{eqnarray*}

Hence, 
$$
\sum_{1\le j\le J} \|M_jf\|_{\ell^2}^2\le 4( 
\sum_{1\le j\le J} \|\widetilde M_jf\|_{\ell^2}^2+ \sum_{N\in I_\rho} 
|f* (L_{N}-K_N)|^2)\, .
$$
 Using \eqref{unifest}, we see that  $\|\widehat L_{[\rho^n]} - \widehat K_{[\rho^n]}\|_\infty\le \frac{C}
{n^{\tilde \gamma}\log \rho}$, with $\tilde \gamma=\min(\gamma,S)>1/2 $.
 Hence 
$$
\sum_{N\in I_\rho} \|f* (L_{N}-K_N)\|_{\ell^2(\Z)}^2 
\le \|f\|_{\ell^2(\Z)}^2 \sum_{N\in I_\rho} \|\widehat L_{N}-\widehat K_N
\|_\infty^2 <\infty\, ;
$$  
 
 Hence, it is enough to prove the theorem with $(\widetilde M_j)$ in place of 
$(M_j)$. 
 Let $t=t(J)$ be an integer to be chosen later. Define 
$R_N$ through its Fourier transform, i.e. 
$$
\widehat{R}_N(x):= \widehat w_{n,0}(x)\eta_0(x) + \sum_{1\le s\le t}
 \, \sum_{2^{s-1}\le q 
<2^s}  \sum_{1\le a\le q,a\wedge q=1} \widehat w_{n,q}(x-a/q) 
\eta_s(x-a/q)\, .
$$
 It follows from \eqref{estim0} that for every 
$1/2<\delta< \tau$, 
\begin{equation*}
\|\sup_{N\in I_\rho }|f*(L_N-R_N)|\|_{\ell^2(\Z)} 
\le C 2^{(\delta-\tau)t}\, .
\end{equation*}

In particular, 
\begin{equation}\label{osc1}
\sum_{1\le j\le J} \|\widetilde M_jf\|_{\ell^2}^2\le 
\sum_{1\le j\le J} \|\sup_{N_j\le N< N_{j+1}, \, N\in I_\rho}
|f* R_{N}-f*R_{N_j}|\|_{\ell^2}^2 + 
CJ 2^{2(\delta-\tau)t}
\end{equation}
Define 
$g_{s,\frac{a}q}$ by $\widehat g_{s,\frac{a}q}(x)= \eta_s(x)f(x+\frac{a}q)$ 
and $g_0(x)=\eta_0(x)f(x)$.
Then, using the change of variable $x\to x+\frac{a}q$, for 
every $k\in\Z$, we have
\begin{eqnarray}\label{osc2}
f*R_N(k)&=&\int_{-1/2}^{1/2}\widehat f(x)  \widehat R_N(x) {\rm e}^{-2i\pi 
kx}\, dx =
\int_{-1/2}^{1/2}\widehat g_0(x) \widehat w_{N,0}(x)  {\rm e}^{-2i\pi 
kx}\, dx \, \cr & &\quad 
+\,\sum_{1\le s\le t}
 \, \sum_{2^{s-1}\le q 
<2^s}  \frac1{q}\sum_{1\le a\le q,a\wedge q=1}  {\rm e}^{-2ik\pi 
\frac{a}q} \int_{-1/2}^{1/2}\widehat g_{s,\frac{a}q}(x)  \widehat w_{N,q}(x) 
{\rm e}^{-2i\pi kx}\, dx\,.
\end{eqnarray}

Hence, 
\begin{eqnarray}\label{osc3}
|f* R_{N}-f*R_{N_j}|&\le& 2^t \max_{1\le s\le t}\max_{{2^{s-1}\le q <2^s\atop  
1\le a\le q}\atop a\wedge q=1} |g_{s,\frac{a}q }*(w_N-w_{N_j})|(k)\, .
\end{eqnarray}
Combining \eqref{osc1}, \eqref{osc2}, \eqref{osc3} and \eqref{oscconv}, 
we infer that 
$$
\sum_{1\le j\le J} \|\widetilde M_jf\|_{\ell^2}^2\le 
C2^{t} \|f\|_{\ell^2(\Z)} + 
CJ 2^{2(\delta-\tau)t}\, ,
$$
which is $o(J)$ if we chose for instance $t(J)=[\log \log J]$, and the 
theorem is proved. 
\end{proof}

%\begin{theorem}\label{aeg}
%Let $(K_n)_{n\ge 1}\subset \ell^1$. Assume that $T_n:=\widehat K_n$ satisfies 
%\eqref{firstcond} and \eqref{secondcond}, for some $\gamma>1/2$. Assume moreover that 
%\eqref{oscconv} holds. Then for every (ergodic) 
%dynamical system $(X, {\mathcal A},\nu,\tau)$ and every $f$ in $L^p$, $p>1/\tau$, the sums
%$$B_nf:=  \sum_{k=1}^\infty  K_n(k) f\circ \tau^k $$ 
%converge  
%$\nu$-a.s.  
%Further   there exists $C_p>0$ such that  $$
%\big\|\sup_{n\ge 1}  B_nf 
%\big\|_p\le C_p\|f\|_p\, .\
%$$
% \end{theorem}
%\begin{proof} Combine with Proposition \ref{maxshift} with Theorem \ref{ae}.
%\end{proof}
%%%%%%%%%%%%
%%%%%%%%%%%%
\section{\bf Proof of Theorem \ref{mt}}     \label{proofmt}

\rm

Firstly, we   prove the dominated ergodic theorem for the weights $(d_n)_{n\ge 1}$. 
In this case, since $D_n$ does not grow too fast, it suffices to deal with positive functions and to prove a 
maximal inequality along the dyadic integers. 

For every $n\ge 2$ and every $q\ge 1$, define 

$$
w_{n,q}:= \frac1{qn\log n}\sum_{1\le k\le n} \log k\, \delta_k +
\frac{2(\gamma-1-\log q)}{n\log n}\sum_{1\le k\le n}  \delta_k\, ,
$$

and 

\begin{equation}\label{psi2}
\psi_{n,q}(x):=\hat w_{n,q}(x)=\frac1{qn}\sum_{1\le k\le n} \log k\, 
{\rm e}^{ikx} +
\frac{2(\gamma-1-\log q)}{n\log n}\sum_{1\le k\le n}  {\rm e}^{ikx}\, .
\end{equation}

Using the well-known estimate $\frac1n|\sum_{1\le k\le n}{\rm e}^{ikx}|
\le \min(1,\frac1{|nx|})$ and Abel summation to deal with the first term in 
 \eqref{psi2}, we see that \eqref{psi} holds for any $\tau \in [0,1)$.
 
 \smallskip
 
 It is also well-known that ({\bf Reference ??}), writing 
 $\kappa_n:= \frac1{n}\sum_{1\le k\le n}  \delta_k$, for every 
 $p>1$, there exists $C_p>0$ such that 
 $$
\|\sup_{n\ge 1} |\kappa_n * g\|_{\ell^p}\le C_p \|g\|_{\ell^p}\qquad 
\forall g\in \ell^r\, . 
 $$
Since $\frac1{n\log n}\sum_{1\le k\le n} \log k\, \delta_k \le 
\kappa_n$, we infer that \eqref{inemaxconv} holds far any $\tau 
\in [0,1)$.

Let $S>1$. For every $n\ge 2$ define
 $$ P_n:= [(\log n )^{3S}[ , \qq Q_n= [n/(\log n )^{2S}]\, .$$
 
 \smallskip
 
 Then, by \eqref{kernelest1} of  Lemma \ref{lemkern} and by Lemma 
 \ref{expdiv1}, we see that \eqref{firstcond} and \eqref{secondcond} 
 holds for $T_n(x):=D_n(x)/D_n$, with $\gamma=S$.

 \medskip
 
 Hence, by Proposition \ref{maxshift} and Calderon's transference principle, 
 we see that $(d_n)_{n\ge 1}$ is a good weight for the dominated ergodic theorem 
 in $L^p$ for every $p\in [1/\tau+\frac{2-1/\tau}{2S},2]$. Since we may take 
 $\tau$ arbitrary close to $1$ and $S$ arbitrary large, the dominated ergodic 
 theorem holds for every $p>1$ as well. 
 
 \medskip
 
 Secondly, we shall prove an oscillation inequality in $L^2$. The proof is 
 exactly as above except that we take $w_{n,q}:= \frac1{qn}\sum_{1\le k\le n}  
  \delta_k$, that we make use \eqref{kernelest2}} of  Lemma \ref{lemkern} 
  and that we apply Theorem \ref{ae} (instead of Proposition 
  \ref{maxshift}). \hfill $\square$

\section{Open Problems.}
  We conclude by listing some  natural problems arising from this work.\vskip 2 pt
\begin{problem}[Extension to $L^{1,\infty}$] Does our main Theorem \ref{mt} remain true in $L^1$? Same question with Theorem \ref{mt1} with  the $\theta $ function in place of the divisor function.
\end{problem}

\begin{problem}[Square function] Let  $\mathcal{ N}=\{n_j, j\ge 1\}$ be an increasing sequence
  of  positive integers, and define for any   $f\in L^2$,
$$
S_\mathcal{ N}(f) =\Big(\sum_{p=1}^\infty  \Vert A^\tau_{n_{j+1}}(f) - A^\tau_{n_{j}}(f)
\Vert_2^2\Big)^{{1/ 2}}, $$
 recalling that $
A_nf:= \frac1{W_n}\sum_{k=1}^n w_k f\circ \tau^k $. When $w_k= d(k)$, is it true that 
$$\|S_\mathcal{ N}(f)\|_2\le C(\mathcal{ N}) \,\|f\|_2$$
for any $f\in L^2$? Is this further true for any increasing sequence of  positive integers 
$\mathcal{ N}$?
\end{problem}

\begin{problem}[Spectral Regularization]
Can one associate to $
V_n(\theta) = \frac1{D_n}\sum_{k=1}^n d(k) \e^{2i\pi n\theta}$, a  regularizing kernel $Q(\theta, y)$ on  $[0,1)^2$ so that for any  
$f\in L^2$, with  spectral measure $\m_f$ (relatively to the operator $Tf=f\circ \tau$),  the new measure defined by
$$
\widehat \mu_f(dy) =  \bigg(\int_{0}^1
Q(\theta,y)\m_f(d\theta)\bigg) d y
$$
verifies  
$$\| A_n^Tf -A_m^Tf\| ^2\le
\widehat \mu_f\Big{\{}\Big]{1\over m},{1\over n}\Big]\Big{\}} ,
$$
for  any integers $m\ge n\ge 1$?  Such an inequality immediately provides a control on the square function associated to these averages.   So is the case for usual ergodic averages where the corresponding oscillations functions can be controlled similarly. We refer to \cite{W} Part I, Section 1.4 concerning this notion and the related results.
\end{problem}
\begin{problem}[Extensions to other arithmetical functions]  Can one  establish the validity of  Theorem \ref{mt} for other arithmetical functions?  Examples can be function $r(n)$ counting the number of ways to write $n$ as a sum of two squares,   the Piltz divisor function $d_k(n)  $ counting the number of ways to write $n$ as a product of $k$ factors (in the latter case   we do not believe  that it is an easy task). In each of these cases, the validity (in $L^1$) of the strong law of large numbers   was recently established in  \cite{BMW}. One may also consider the same question for the  multiplicative function  $R(u)= \#\{ (\d, d)\in \N^2 : [d,\d]=u\}$. \end{problem}


\begin{thebibliography}{99}

 
  \bibitem{BC}  P. T. Bateman, S. Chowla,    \emph{Some special trigonometrical series related to the distribution of primes}, J. London Math. Soc.,  {\bf 38}, 372--374, (1963).
 \bibitem{BMW}  I. Berkes, W. M\"uller, M. Weber,     \emph{On the strong law of large numbers and arithmetic functions}, Indagationes Math. {\bf 23}, 547--555,  (2012).
 \bibitem{B} J. Bourgain,   \emph{On the maximal ergodic theorem for certain subsets of the integers}, Israel J. Math. 61, 39-72, (1988).
\bibitem{B1} 
J. Bourgain,  
\emph{Pointwise ergodic theorems for arithmetic sets, with an appendix
 on return time sequences}, jointly with H. Furstenberg, Y. Katznelson,
D. Ornstein, \emph{Inst. Hautes \'Etudes Sci. Publ. Math.} \textbf{69},  5--45, (1988).
\bibitem{B2} J. Bourgain, \emph{An approach to pointwise ergodic theorems},  Geometric aspects of functional analysis (1986/87),  204-223, Lecture Notes in Math., {\bf 1317}, Springer, Berlin, (1988).
\bibitem{Davenport}  H. Davenport,  \emph{On some infnite series
 involving arithmetical
functions},  Quart. J. Math., Oxf. Ser. \textbf{8},  8-13, (1937).
\bibitem{Davenport2} H. Davenport,  \emph{On some infnite 
series involving arithmetical
functions II}, Quart. J. Math., Oxf. Ser.  \textbf{8},   313-320, (1937).
%\bibitem{Co} J. E. Collison    \emph{Central moments for arithmetic functions}, Pacific
%J. Math. 77 No 2, 307--314, (1977).
% \bibitem{CMS}  B. Crstici, D. S. Mitrinovi\'c, J. S\'andor     {\it Handbook of Number Theory I},   %Second Edition, (2006),Springer, Dordrecht, The Netherlands.
\bibitem{Delange} H. Delange, \emph{Sur des formules de Atle Selberg}, 
Acta Arith. \textbf{19}, 105-146, (1971).
\bibitem{DQ} C. Demeter and A. Quas, \emph{Weak-$L^1$-estimates and ergodic 
theorems}, New York J. Math., {\bf 10}, 169-174, (2004).
% \bibitem{Gr} T.H. Gronwall    (1912){\sl Some asymptotic expressions in the theory of %numbers},  Trans. Amer. Math. Soc. {\bf 8},  118--122.
 \bibitem{HL} G. Hardy,  J. E.  Littlewood    (1930)   \emph{A maximal theorem with functions-theoretic applications},  Acta Math. {\bf
 54}, 81--116.
 \bibitem{JKRW} R.
Jones, R. Kaufman, J. Rosenblatt and  M. Wierdl, \emph{Oscillation in ergodic theory}, Ergodic Theory Dynam. Systems   \textbf{18}    no. 4, 889-935,
(1998).
  \bibitem{J}  M. Jutila, \emph{On exponential sums involving the divisor function.},  J. Reine Angew. Math. {\bf 355}, 173-190,  (1985).
% \bibitem{KN} L.  Kuipers,   and H. Niederreiter,   (1974)
%{\it Uniform distribution of sequences}, Pure  Appl. Math., Wiley-Interscience, New York.
%\bibitem{LW} M.   Lin,   M.  Weber   \emph{Weighted ergodic theorems and
% strong laws of large numbers}, {\it Ergodic Theory Dynam. Systems}  \textbf{27}, 511--543, %(2007).
\bibitem{MT} M. Mirek and B. Trojan, \emph{Cotlar's ergodic theorem along the prime numbers}, arXiv:1311.7572.
%\bibitem{MV} {H. L.  Montgomery, R. C.  Vaughan} (2007) {\it Multiplicative Number Theory I. %Classical Theory}, Cambridge studies in advanced math, Cambridge University Press, %Cambridge, United Kingdom.
\bibitem{Sarnak} P. Sarnak, \emph{Three lectures on the M�obius function randomness and dynamics}, publications.ias.edu/sarnak/. 
\bibitem{SW} E. M. Stein and N. J. Weiss, \emph{On the convergence of Poisson integrals}, Trans. Amer. Math. Soc. {\bf 140} (1969), 35-54.

 \bibitem{Tenenbaum} G. Tenenbaum,  \emph{Introduction to analytic and probabilistic number theory}. Translated from the second French edition (1995) by C. B. Thomas. Cambridge Studies in Advanced Mathematics,  \textbf{46}. Cambridge University Press, Cambridge, 1995. xvi+448 pp.
%\bibitem{T} {Tenenbaum G.} (2008) {\it Introduction \`a la th\'eorie analytique et probabiliste des %nombres}, Coll. \'Echelles Ed. Belin  Paris.
  \bibitem{W}  M. Weber (2009)   {\sl Dynamical Systems and Processes}, European Mathematical
Society Publishing House, IRMA Lectures
 in Mathematics 
and Theoretical Physics {\bf 14}, xiii+759p. 
\bibitem{Wierdl} M. Wierdl, \emph{Pointwise ergodic theorem along the prime numbers}, Israel J. Math.   \textbf{64}    no. 3, 315-336, (1989).
 \bibitem{Wil} J. R.  Wilton,     \emph{An approximate functional equation with applications to a problem of Diophantine approximation}, J. f\"ur die reine und angewandte Math.  \textbf{169}, 219--237, (1933). 
\bibitem{Wi}  A. Wintner, \emph{The theory of measure in arithmetical semi-groups},  Wawerly Press, Baltimore, Md. (1944).
  \bibitem{Z}    A. Zygmund,   (2002)   \emph{Trigonometric series}, Third Ed. Vol. {\bf  1}\&{\bf  2} combined,
Cambridge Math. Library, Cambridge Univ. Press. 

    \end{thebibliography}
\end{document}